\documentclass{imanum}
\jno{drnxxx}
\received{23 June 2016}
\revised{}

\usepackage{graphicx,graphics,color,booktabs}
\usepackage{epstopdf}

\usepackage{color}

\newcommand\bfx{{\mathbf x}}

\newcommand\bfA{{\mathbf A}}

\newcommand\bfK{{\mathbf K}}
\newcommand\bfM{{\mathbf M}}

\newcommand\andquad{\quad\hbox{ and }\quad}

\renewcommand\d{\hbox{\rm d}}

\newcommand{\diff}{\frac{\textrm{d}}{\textrm{d}t}}
\newcommand{\Ga}{\Gamma}

\newcommand{\Gat}{\Gamma(t)}
\newcommand{\GT}{\mathcal{G}_T}
\newcommand{\laplace}{\Delta}

\newcommand{\nbg}{\nabla_{\Gamma}}
\newcommand{\nbgh}{\nabla_{\Gamma_h}}
\newcommand{\mat}{\partial^{\bullet}}
\newcommand{\Pt}{\widetilde{\mathcal{P}}_h}

\newcommand{\co}{continuous}
\renewcommand{\d}{\textrm{d}}

\newcommand{\inv}{^{-1}}

\newcommand{\N}{\mathbb{N}}
\newcommand{\nb}{\nabla}

\newcommand{\pa}{\partial}
\newcommand{\R}{\mathbb{R}}

\newcommand{\spn}{\textnormal{span}}
\newcommand{\st}{such that}

\def \t {(t)}
\newcommand{\Th}{\mathcal{T}_h}
\def \to {\rightarrow}
\newcommand{\vphi}{\varphi}

\usepackage{marginnote}

\newcommand\enodes\xs
\newcommand\nnodes\bfx
\newcommand\rnodes\xs

\newcommand\regmass\bfK
\newcommand\mass\bfM
\newcommand\stiff\bfA

\newcommand\abs[1]{\lvert #1\rvert}

\newcommand{\Btensor}{\mathcal{B}}

\newcommand{\xs}{\bfx^\ast}

\newcommand{\aast}{a^{\ast}}
\def \P {\mathcal{P}_h}

\renewcommand{\nu}{\mathrm{n}} 

\begin{document}

\title{High-order evolving surface finite element method \\ for parabolic problems on evolving surfaces}
\shorttitle{High-order ESFEM for evolving surface PDEs}

\author{%
{\sc Bal\'{a}zs Kov\'{a}cs\thanks{Corresponding author. Email: kovacs@na.uni-tuebingen.de}} \\
Mathematisches Institut\\ University of T\"{u}bingen\\ Auf der Morgenstelle 10\\ 72076 T\"{u}bingen, Germany
}

\shortauthorlist{B.~Kov\'{a}cs}

\maketitle

\begin{abstract}
  {High-order spatial discretisations and full discretisations of parabolic partial differential equations on evolving surfaces are studied. We prove convergence of the high-order evolving surface finite element method, by showing high-order versions of geometric approximation errors and perturbation error estimates and by the careful error analysis of a modified Ritz map. Furthermore, convergence of full discretisations using backward difference formulae and implicit Runge--Kutta methods are also shown.}
  {parabolic problems, evolving surfaces, high-order ESFEM, Ritz map, convergence, BDF and Runge--Kutta methods}

\end{abstract}

%

\section{Introduction}


Numerical methods for partial differential equations (PDEs) on stationary and evolving surfaces and for coupled bulk--surface PDEs are under intensive research in the recent years. Surface finite element methods are all based on the fundamental paper of \cite{Dziuk88}, further developed for evolving surface parabolic problems by \cite{DziukElliott_ESFEM,DziukElliott_L2}.

High-order versions of various finite element methods for problems on a \emph{stationary surface} have received attention in a number of publications previously. We give a brief overview on this literature here:
\begin{itemize}
  \item The surface finite element method of \cite{Dziuk88} was extended to higher order finite elements on stationary surfaces by \cite{Demlow2009}. Some further important results for higher order surface finite elements were shown by \cite{ElliottRanner}.
  \item Discontinuous Galerkin method for elliptic surface problems were analysed by \cite{DednerMadhavanStinner}, and then extended to high-order discontinuous Galerkin methods  in \cite{highorderDG}.
  \item Recently, the \emph{unfitted} (also called trace or cut) finite element methods are investigated intensively, see for instance \cite{OlshanskiiReuskenGrande}, \cite{Reusken} and \cite{cutFEM}. A higher order version of the trace finite element method was analysed by \cite{GrandeReusken}.
\end{itemize}
However, to our knowledge, there are no papers showing convergence of the \emph{high-order} evolving surface finite element method for parabolic partial differential equations on \emph{evolving surfaces}.

\smallskip
In this paper we extend the $H^1$ and $L^2$ norm convergence results of \cite{DziukElliott_ESFEM} and \cite{DziukElliott_L2} to \emph{high-order} evolving surface finite elements applied to parabolic problems on evolving surfaces with prescribed velocity. Furthermore, convergence results for full discretisations using Runge--Kutta methods, based on \cite{DziukLubichMansour_rksurf}, and using backward difference formulae (BDF), based on \cite{LubichMansourVenkataraman_bdsurf}, are also shown.

To prove high-order convergence of the spatial discretisation three main groups of errors have to be analysed: \begin{itemize}
  \item \emph{Geometric errors}, resulting from the appropriate approximation of the smooth surface. Many of these results carry over from \cite{Demlow2009} by careful investigation of time dependencies, while others are extended from \cite{diss_Mansour} and \cite{DziukLubichMansour_rksurf}, \cite{LubichMansourVenkataraman_bdsurf}.
  \item \emph{Perturbation errors of the bilinear forms}, whose higher order version can be shown by carefully using the core ideas of \cite{DziukElliott_L2}.
  \item High-order estimates for the \emph{errors of a modified Ritz projection}, which was defined in \cite{LubichMansour_wave}. These projection error bounds rely on the nontrivial combination of the mentioned geometric error bounds and on the well known Aubin--Nitsche trick.
\end{itemize}

We further show convergence results for full discretisations using Runge--Kutta and BDF methods. The error estimates for Runge--Kutta methods shown in \cite{DziukLubichMansour_rksurf} and for BDF methods in \cite{LubichMansourVenkataraman_bdsurf} are applicable without any modifications, since, the semidiscrete problem can be written in a matrix-vector formulation, cf.\ \cite{DziukLubichMansour_rksurf}, where the matrices have exactly the same properties as in the linear finite element case. Therefore, the fully discrete convergence results transfer to high-order evolving surface finite elements using the mentioned error estimates of the Ritz map.

The implementation of the high-order method is also a nontrivial task. The matrix assembly of the time dependent mass and stiffness matrices is based on the usual reference element technique. Similarly to isoparametric finite elements, the approximating surface is parametrised over the reference element, hence, all the computations are done there.

It was pointed out by \cite{GrandeReusken}, that the approach of \cite{Demlow2009} requires \textit{explicit knowledge of the exact signed distance function to the surface $\Ga$}. However, the signed distance function is only used in the analysis, but it is not required for the computations away from the initial time level. It is only used for generating the initial surface approximation.

Here we only consider linear parabolic PDEs on evolving surfaces, however we believe our techniques and results carry over to other cases, such as to the Cahn--Hilliard equation \cite{ElliottRanner_CH}, to wave equations \cite{LubichMansour_wave}, to ALE methods \cite{ElliottVenkataraman_ALE}, \cite{KovacsPower_ALEdiscrete}, nonlinear problems \cite{KovacsPower_quasilin} and to evolving versions of bulk--surface problems \cite{ElliottRanner}. Furthermore, while in this paper we only consider evolving surfaces with prescribed velocity, many of the high-order geometric estimates of this paper are essential for the numerical analysis of parabolic problems where the \emph{surface velocity depends on the solution}, cf.~\cite{soldriven}.

The paper is organised in the following way:
In Section~\ref{section: the problem} we recall the basics of linear parabolic problems on evolving surfaces together with some notations based on \cite{DziukElliott_ESFEM}.
Section~\ref{section: ESFEM} deals with the description of higher order evolving surface finite elements based on \cite{Demlow2009}.
Section~\ref{section: main result} contains the time discretisations and states the main results of this paper: semidiscrete and fully discrete convergence estimates.
In Section~\ref{section: geometric estimates} we turn to the estimates of geometric errors and geometric perturbation estimates.
In Section~\ref{section: Ritz} the errors in the generalised Ritz map and in its material derivatives are estimated.
Section~\ref{section: error bounds} contains the proof of the main results.
Section~\ref{section: implementation} shortly describes the implementation of the high-order evolving surface finite elements. 
In Section~\ref{section: numerics} we present some numerical experiments illustrating our theoretical results.

\section{The problem}
\label{section: the problem}

Let us consider a smooth evolving compact surface $\Ga\t \subset \R^{m+1}$ ($m\leq 2$), $0 \leq t \leq T$, which moves with a given smooth velocity $v$. Let $\mat u = \pa_{t} u + v \cdot \nb u$ denote the material derivative of the function $u$, where $\nbg$ is the tangential gradient given by $\nbg u = \nb u -\nb u \cdot \nu \nu$, with unit normal $\nu$. We are sharing the setting of \cite{DziukElliott_ESFEM}, \cite{DziukElliott_L2}.

We consider the following linear problem on the above surface, for $u=u(x,t)$,
\begin{equation}
\label{eq: strong form}
    \begin{cases}
    \begin{alignedat}{4}
        \mat u + u \nb_{\Gat} \cdot v - \laplace_{\Gat} u &= f & \qquad & \textrm{ on } \Ga\t ,\\
        u(\cdot,0) &= u_0 & \qquad & \textrm{ on } \Ga(0),
    \end{alignedat}
    \end{cases}
\end{equation}
where for simplicity we set $f=0$, but all our results hold with a non-vanishing inhomogeneity as well.

Let us briefly recall some important concepts used later on, whereas in general for basic formulas we refer to \cite{DziukElliott_acta}.
An important tool is the Green's formula on closed surfaces
\begin{equation*}
    \int_{\Gat} \nb_{\Gat} z \cdot \nb_{\Gat} \phi = -\int_{\Gat} (\laplace_{\Gat} z) \phi.
\end{equation*}
We use Sobolev spaces on surfaces: For a smooth surface
$\Gat$, for fixed $t\in[0,T]$, we define
$$
    H^{1}(\Gat) = \bigl\{ \eta \in L^{2}(\Gat) \mid \nb_{\Gat} \eta \in
    L^{2}(\Gat)^{m+1} \bigr\},
$$
and analogously $H^{k}(\Gat)$ for $k\in \N$, cf.\ Section~2.1 of \cite{DziukElliott_ESFEM}.
Finally, the space-time manifold is denoted by $\GT= \cup_{t\in[0,T]} \Ga\t\times\{t\}$.

The variational formulation of this problem reads as: Find $u\in H^1(\GT)$ such that
\begin{equation}
\label{eq: weak form}
    \diff \int_{\Gat} u \vphi + \int_{\Gat} \nb_{\Gat} u \cdot \nb_{\Gat} \vphi = \int_{\Gat} u\mat \vphi
\end{equation}
holds for almost ever $t\in(0,T)$ for every $\vphi(\cdot,t) \in H^1(\Gat)$, and $u(\cdot,0)=u_0$ holds. For suitable $u_{0}$ existence and uniqueness results for \eqref{eq: weak form} were obtained by Theorem~4.4 in \cite{DziukElliott_ESFEM}.

\section{High-order evolving surface finite elements}
\label{section: ESFEM}

We define the high-order evolving surface finite element method (ESFEM) applied to our problem following \cite{Demlow2009} and \cite{Dziuk88}, \cite{DziukElliott_ESFEM}. We use simplicial elements and continuous piecewise polynomial basis functions of degree $k$.

\subsection{Basic notions}

The smooth \emph{initial} surface $\Ga(0)$ is approximated by a $k$-order interpolating discrete surface denoted by $\Ga_h^k(0)$, with vertices $a_i(0)$, $i=1,2,\dotsc,N$, and it is given as
\begin{equation*}
    \Ga_h^k(0) = \bigcup_{E(0)\in \Th^k(0)} E(0).
\end{equation*}
More details and the properties of such a discrete initial surface can be found in Section~2 of \cite{Demlow2009}.

Then the nodes of the initial triangulation, sitting on the exact surface, are moved along the space-time manifold with the known velocity $v$. Hence, the nodes $(a_i(t))_{i=1}^N$ define a triangulation also for $0 < t \leq T$, by
\begin{equation*}
    \Ga_h^k(t) = \bigcup_{E(t)\in \Th^k(t)} E(t).
\end{equation*}
The discrete surface $\Ga_h^k(t)$ remains to be an interpolation of $\Gat$ for all times. We always assume that the evolving elements $E(t)$ are forming an admissible triangulation $\Th^k(t)$, with $h$ denoting the maximum diameter, cf.\ Section~5.1 of \cite{DziukElliott_ESFEM}. The discrete tangential gradient on the discrete surface $\Ga_h^k\t$ is given by
\begin{equation*}
    \nb_{\Ga_h^k\t} \phi := \nb {\phi} - \nb {\phi} \cdot \nu_h \nu_h,
\end{equation*}
understood in an elementwise sense, with $\nu_h$ denoting the normal to $\Ga_h^k(t)$, see \cite{Demlow2009}, \cite{DziukElliott_ESFEM}.

For every $t\in[0,T]$ and for the $k$-order interpolation $\Ga_h^k\t$ we define the finite element subspace of order $k$, denoted by $S_h^k\t$, spanned by the continuous evolving basis functions $\chi_j$ of piecewise degree $k$, satisfying $\chi_j(a_i\t,t) = \delta_{ij}$ for all $i,j = 1, 2, \dotsc, N$, therefore
\begin{equation*}
    S_h^k\t = \spn\big\{ \chi_1(  \cdot ,t), \chi_2(  \cdot ,t), \dotsc, \chi_N(  \cdot ,t) \big\}.
\end{equation*}

\noindent We interpolate the surface velocity on the discrete surface using the basis functions and denote it with $V_h$. Then the discrete material derivative is given by
\begin{equation*}
    \mat_h \phi_h = \pa_t \phi_h + V_h \cdot \nb \phi_h  \qquad (\phi_h \in S_h^k\t).
\end{equation*}
The key \textit{transport property} derived in Proposition 5.4 in \cite{DziukElliott_ESFEM}, is the following
\begin{equation}
\label{eq: transport property}
    \mat_h \chi_j = 0 \qquad \textrm{for} \quad j=1,2,\dotsc,N.
\end{equation}

The spatially discrete problem (for a fixed degree $k$) then reads as: Find $U_h\in S_h^k\t$ such that
\begin{equation}
\label{eq: semidiscrete problem}
        \diff \int_{\Ga_h^k\t} U_h \phi_h
        + \int_{\Ga_h^k\t} \nbgh U_h \cdot \nbgh \phi_h = \int_{\Ga_h^k\t}\!\! U_h \mat_h \phi_h,  \qquad (\forall \phi_h \in S_h^k\t),
\end{equation}
with the initial condition $U_h^0\in S_h^k(0)$ being a suitable approximation to $u_0$.

Later on we will always work with a high-order approximation surface $\Ga_h^k\t$ and with the corresponding high-order evolving surface finite element space $S_h^k\t$ (with $2\leq k\in\N$), therefore from now on we drop the upper index $k$, unless we would like to emphasize it, or it is not clear from the context.

\subsection{The ODE system}
Similarly to \cite{DziukLubichMansour_rksurf}, the ODE form of the above problem \eqref{eq: semidiscrete problem} can be derived by setting
\begin{equation*}
    U_h( \, \cdot \,,t) = \sum_{j=1}^N \alpha_j\t \chi_j( \, \cdot \,,t)
\end{equation*}
into the semidiscrete problem, and by testing with $\phi_h=\chi_j$ ($j=1,2,\dotsc,N$), and using the transport property \eqref{eq: transport property}.

The spatially semidiscrete problem \eqref{eq: semidiscrete problem} is equivalent to the following ODE system for the vector $\alpha\t=(\alpha_j\t)_{j=1}^N \in\R^N$, collecting the nodal values of $U_h(\cdot,t)$,
\begin{equation}\label{eq: ODE system}
    \begin{alignedat}{2}
        \diff \Big(M\t \alpha\t\Big) + A\t \alpha\t &= 0, \\
        \alpha(0) &= \alpha_0,
    \end{alignedat}
\end{equation}
where the evolving mass matrix $M\t$ and the stiffness matrix $A\t$ are defined as
\begin{equation*}
    M(t)|_{kj} = \int_{\Ga_h\t} \chi_j \chi_k \andquad A\t|_{kj} =
    \int_{\Ga_h\t} \nbgh \chi_j \cdot \nbgh  \chi_k ,
\end{equation*}
for $j,k = 1,2, \dotsc, N$.

\subsection{Lifts}
For the error analysis we need to compare functions on different surfaces. This is conveniently done by the
 \emph{lift operator}, which was introduced in \cite{Dziuk88} and further investigated in \cite{DziukElliott_ESFEM}, \cite{DziukElliott_L2}. The lift operator maps a function on the discrete surface onto a function on the exact surface.

Let $\Ga_h(t)$ be a $k$-order approximation to the exact surface $\Ga(t)$. Using the oriented distance function $d$ (cf.\ Section~2.1 of \cite{DziukElliott_ESFEM}), the lift of a continuous function $\eta_h \colon \Ga_h(t) \to \R$ is defined as
\begin{equation*}
    \eta_{h}^{l}(p,t) := \eta_h(x,t), \qquad x\in\Ga(t),
\end{equation*}
where for every $x\in \Ga_h(t)$ the point $p=p(x,t)\in\Ga(t)$ is uniquely defined via
\begin{equation}
\label{eq: lift defining equation}
    p = x - \nu(p,t) d(x,t).
\end{equation}
By $\eta^{-l}$ we denote the function on $\Ga_h(t)$ whose lift is $\eta$.

In particular we will often use the space of lifted basis functions
\begin{equation*}
    (S_h\t)^l=(S_h^k\t)^l = \big\{ \phi_h^l \ | \ \phi_h \in S_h^k\t \big\}.
\end{equation*}

\section{Convergence estimates}
\label{section: main result}

\subsection{Convergence of the semidiscretisation}
We are now formulate the convergence theorem for the semidiscretisation using high-order evolving surface finite elements. This result is the higher order extension of Theorem~4.4 in \cite{DziukElliott_L2}.

\begin{theorem}
\label{theorem: semidiscrete convergence}
    Consider the evolving surface finite element method of order $k$ as space discretisation of the parabolic problem \eqref{eq: strong form}. Let $u$ be a sufficiently smooth solution of the problem, and assume that the initial value satisfies
    \begin{equation*}
        \| u_h^0 - u(\cdot,0) \|_{L^2(\Ga(0))} \leq C_0 h^{k+1}.
    \end{equation*}
    Then there exists $h_0>0$ \st\ for $h\leq h_0$, the following error estimate holds for $t \leq T$:
    \begin{equation*}
        \|u_h(\cdot,t) - u(\cdot,t)\|_{L^2(\Ga(t))} + h\Big( \int_0^t \|\nb_{\Ga(s)} (u_h(\cdot,s) - u(\cdot,s)\|_{L^2(\Ga(s))}^2 \d s  \Big)^{\frac{1}{2}} \leq C h^{k+1} .
    \end{equation*}
    The constant $C$ is independent of $h, \ \tau$ and $n$, but depends on $T$.
\end{theorem}
The proof of this result is postponed to a later section, after we have shown some preparatory results.

\subsection{Time discretisation: Backward differentiation formulae}

We apply a $p$-step backward difference formula (BDF) for $p\leq5$ as a discretisation to the ODE system \eqref{eq: ODE system}, coming from the ESFEM space discretisation of the parabolic evolving surface PDE.

We briefly recall the $p$-step BDF method applied to the system \eqref{eq: ODE system} with step size $\tau>0$:
\begin{equation}
\label{def: BDF}
    \frac{1}{\tau} \sum_{j=0}^p \delta_j M(t_{n-j})\alpha_{n-j} + A(t_n)\alpha_n = 0, \qquad (n \geq p),
\end{equation}
where the coefficients of the method are given by $\delta(\zeta)=\sum_{j=0}^p \delta_j \zeta^j=\sum_{\ell=1}^p \frac{1}{\ell}(1-\zeta)^\ell$, while the starting values are $\alpha_0, \alpha_1, \dotsc, \alpha_{p-1}$. The method is known to be $0$-stable for $p\leq6$ and have order $p$ (for more details, see Chapter~V. of \cite{HairerWannerII}).

In the following result we compare the fully discrete solution
\begin{equation*}
    U_h^n = \sum_{j=1}^N \alpha_j^n \chi_j( \, \cdot \,,t_n),
\end{equation*}
obtained by solving \eqref{def: BDF} and the Ritz map $\Pt : H^{1}(\Gat) \to S_h^k\t$ ($t\in[0,T]$) of the sufficiently smooth solution $u$. The precise definition of the Ritz map is given later.
The following error bound was shown in \cite{LubichMansourVenkataraman_bdsurf} for BDF methods up to order five. See also \cite{diss_Mansour}.
\begin{theorem}[\cite{LubichMansourVenkataraman_bdsurf}, Theorem~5.1]
\label{theorem: BDF error estimates}
    Consider the parabolic problem \eqref{eq: strong form}, having a sufficiently smooth solution for $0\leq t \leq T$. Couple the $k$-order evolving surface finite element method as space discretisation with time discretisation by a $p$-step backward difference formula with $p\leq5$. Assume that the Ritz map of the solution has continuous discrete material derivatives up to order $p+1$. Then there exists $\tau_0>0$, independent of $h$, \st\ for $\tau \leq \tau_0$, for the error $E_h^n=U_h^n-\Pt u(\cdot,t_n)$ the following estimate holds for $t_n=n\tau \leq T$:
    \begin{align*}
        \|E_h^n\|_{L^2(\Ga_h(t_n))} +& \Big( \tau \sum_{j=1}^n \|\nb_{\Ga_h(t_j)} E_h^j \|_{L^2(\Ga_h(t_j))}^2 \Big)^{\frac{1}{2}} \\
        \leq C \tilde{\beta}_{h,p} \tau^{p} +& \Big( \tau \sum_{j=1}^n \|R_h(\cdot,t_j) \|_{H^{-1}_{h}(\Ga_h(t_j))}^2 \Big)^{\frac{1}{2}} + C \max_{0\leq i \leq p-1} \|E_h^i\|_{L^2(\Ga_h(t_i))},
    \end{align*}
    where the constant $C$ is independent of $h, n$ and $\tau$, but depends on $T$. Furthermore
    \begin{equation*}
        \tilde{\beta}_{h,p}^2 = \int_0^T \sum_{\ell=1}^{p+1} \| (\mat_h)^{(\ell)} (\Pt u)(\cdot,t) \|_{L^2(\Ga_h(t))} \d t.
    \end{equation*}
    The $H_h\inv(\Ga_h\t)$-norm, of the ESFEM residual $R_h$, is defined as
    \begin{equation*}
        \|R_h(\cdot,t) \|_{H_h^{-1}(\Ga_h\t)} = \sup_{0\neq\phi_h\in S_h^k\t} \frac{m_h(R_h(\cdot,t),\phi_h)}{\|\phi_h\|_{H^{1}(\Ga_h\t)}} \ .
    \end{equation*}
\end{theorem}

\begin{remark}
    The most important technical tools in the proof of Theorem~\ref{theorem: BDF error estimates}, and also in the proof of the BDF stability result Lemma~4.1 in \cite{LubichMansourVenkataraman_bdsurf}, are the ODE formulation \eqref{eq: ODE system}, and the key estimates first shown in Lemma~4.1 in \cite{DziukLubichMansour_rksurf}. In this result the following estimates are shown: there exist $\mu,\kappa>0$ such that, for $w,z \in \R^N$,
    \begin{equation*}
        w^T \big( M(s) - M\t\big) z \leq (e^{\mu(s-t)}-1) \|w\|_{M\t}\|z\|_{M\t} , \quad
        w^T \big( A(s) - A\t\big) z \leq (e^{\kappa(s-t)}-1) \|w\|_{A\t}\|z\|_{A\t} .
    \end{equation*}
    The proof of these bounds only involve basic properties of the mass and stiffness matrix $M\t$ and $A\t$, hence it is independent of the order of the basis functions. Therefore, the high-order ESFEM versions of the above inequalities also hold. Therefore the original results of \cite{LubichMansourVenkataraman_bdsurf} also hold here.
\end{remark}

\subsection{Convergence of the full discretisation}

We are now in the position to formulate one of the main results of this paper, which yields optimal-order error bounds for high-order finite element semidiscretisation coupled to BDF methods up to order five applied to an evolving surface PDE.

\begin{theorem}[$k$-order ESFEM and BDF-$p$]
\label{theorem: ESFEM BDF error}
    Consider the evolving surface finite element method of order $k$ as space discretisation of the parabolic problem \eqref{eq: strong form}, coupled to the time discretisation by a $p$-step backward difference formula with $p\leq5$. Let $u$ be a sufficiently smooth solution of the problem, and assume that the starting values are satisfying (with $\P u=(\Pt u)^l)$)
    \begin{equation*}
        \max_{0 \leq i \leq p-1} \| u_h^i - (\P u)(\cdot,t_i) \|_{L^2(\Ga(t_i))} \leq C_0 h^{k+1}.
    \end{equation*}
    Then there exists $h_0>0$ and $\tau_0>0$, \st\ for $h\leq h_0$ and $\tau \leq \tau_0$, the following error estimate holds for $t_n=n\tau \leq T$:
    \begin{equation*}
        \|u_h^n - u(\cdot,t_n)\|_{L^2(\Ga(t_n))} + h\Big( \tau \sum_{j=1}^n \|\nb_{\Ga(t_j)} (u_h^j - u(\cdot,t_j))\|_{L^2(\Ga(t_j))}^2 \! \Big)^{\frac{1}{2}} \leq C \big( \tau^{p} \! + h^{k+1} \big).
    \end{equation*}
    The constant $C$ is independent of $h, \ \tau$ and $n$, but depends on $T$.
\end{theorem}
The proof of this result is also postponed to a later section, after we have shown some preparatory results.

\begin{remark}
\label{remark: Runge--Kutta extension}
    We remark here that an analogous fully discrete convergence result is readily available for algebraically stable implicit Runge--Kutta methods (such as the Radau IIA methods), since the Runge--Kutta analog of Theorem~\ref{theorem: BDF error estimates} have been proved in \cite{DziukLubichMansour_rksurf}. The combination of this result with our high-order semidiscrete error bounds (Theorem~\ref{theorem: res bound}) proves the Runge--Kutta analog of Theorem~\ref{theorem: ESFEM BDF error}.
\end{remark}

\section{Geometric estimates}
\label{section: geometric estimates}

In this section we present further notations and some technical lemmas that will be used later on in the proofs leading to the convergence result. These estimates are high-order analogs of some previous results proved in \cite{Dziuk88}, \cite{DziukElliott_ESFEM}, \cite{DziukElliott_L2}, \cite{Demlow2009} and \cite{diss_Mansour}.

\subsection{Geometric approximation results}
\label{section: lift}

\newcommand{\pr}{\textnormal{Pr}}
\newcommand{\Id}{\textnormal{Id}}
\newcommand{\wein}{\mathcal{H}}

In the following we state and prove estimates for the errors resulted by the geometric surface approximation.

\begin{lemma}[equivalence of norms, \cite{Dziuk88}, \cite{Demlow2009}]
\label{lemma: equivalence of discrete norms}
    Let $\eta_h : \Ga_h(t) \to \R$ with lift $\eta_h^l : \Ga(t) \to \R$. Then the discrete and continuous $L^p$- and Sobolev norms are equivalent, independently of the mesh size $h$.
\end{lemma}
For instance, there is a constant $c>0$ such that for all $h$:
\begin{align*}
    c\inv \|\eta_h\|_{L^2(\Ga_h(t))} \leq &\ \|\eta_h^l\|_{L^2(\Ga(t))} \leq c\|\eta_h\|_{L^2(\Ga_h(t))} , \\
    c\inv \|\eta_h\|_{H^1(\Ga_h(t))} \leq &\ \|\eta_h^l\|_{H^1(\Ga(t))} \leq c\|\eta_h\|_{H^1(\Ga_h(t))} .
\end{align*}

We now turn to the study of some geometric concepts and their errors. By $\delta_h$ we denote the quotient between the \co\ and discrete surface measures, $\d A$ and $\d A_h$, defined as $\delta_h \d A_h = \d A$. Further, we recall that
$\pr$ and $\pr_h$ are the projections onto the tangent spaces of $\Ga(t)$ and $\Ga_h(t)$, respectively.
We further set, from \cite{DziukElliott_L2},
\begin{equation*}
    Q_h = \frac{1}{\delta_h} (\Id - d\wein)\pr\pr_h\pr(\Id - d\wein),
\end{equation*}
where $\wein$ ($\wein_{ij} = \pa_{x_j}\nu_i$) is the (extended) Weingarten map. Using this notation and \eqref{eq: lift defining equation}, in the proof of Lemma~5.5 \cite{DziukElliott_L2} it is shown that
\begin{equation}
\label{eq: gradient transformation}
    \nbgh \phi_h(x) \cdot \nbgh \phi_h(x) = \delta_h Q_h \nbg \phi_h^l(p) \cdot \nbg \phi_h^l(p).
\end{equation}

For these quantities we show some results analogous to their linear ESFEM version showed in Lemma 5.4 in \cite{DziukElliott_L2}, Proposition 2.3 in \cite{Demlow2009}, or Lemma~6.1 in \cite{diss_Mansour}.

Later on the following estimates will play a key technical role.
\begin{lemma}
\label{lemma: geometric est}
For $\Ga_h^k\t$ and $\Ga\t$ as above, we have the geometric approximation estimates:
    \begin{align*}
        &\ \|d\|_{L^\infty(\Ga_h(t))} \leq c h^{k+1} , \qquad
        &&\ \|1-\delta_h\|_{L^\infty(\Ga_h(t))} \leq c h^{k+1} , \\
        &\ \|\nu - \nu_h\|_{L^\infty(\Ga_h(t))} \leq c h^{k}  \qquad
        &&\ \textrm{and} \\
        &\ \|\Id - \delta_h Q_h\|_{L^\infty(\Ga_h(t))} \leq c h^{k+1} , \qquad
        &&\ \|(\mat_h)^{(\ell)}d\|_{L^\infty(\Ga_h(t))} \leq c h^{k+1} , \\
        &\ \|(\mat_h)^{(\ell)} \delta_h\|_{L^\infty(\Ga_h(t))} \leq c h^{k+1} , \qquad
        &&\ \|\pr((\mat_h)^{(\ell)}Q_h)\pr\|_{L^\infty(\Ga_h)} \leq ch^{k+1} .
    \end{align*}
    with constants depending only on $\GT$, but not on $h$ or $t$.
\end{lemma}
\begin{proof}
The first three bounds were shown in Proposition 2.3 of \cite{Demlow2009} for the stationary case. Noting that the constants only depend on $\GT$ these inequalities are shown.

The last four bounds are simply the higher order extensions of the corresponding estimates in Lemma~6.1 of \cite{diss_Mansour}. They can be proved in the exact same way using the bounds of the first three estimates.
\end{proof}

\subsection{Interpolation estimates for evolving surface finite elements}

The following result gives estimates for the error in the interpolation. Our setting follows that of Section 2.5 of \cite{Demlow2009}.

Let us assume that the surface $\Ga(t)$ is approximated by the interpolation surface $\Ga_h^k(t)$. Then for any $w\in H^{k+1}(\Ga(t))$, there is a unique $k$-order surface finite element interpolation $\widetilde I_h^k w \in S_h^k(t)$, furthermore we set $(\widetilde I_h^k w)^l = I_h^k w$.

\newcommand{\Ih}{\widetilde{I}_h}
\begin{lemma}[\cite{Demlow2009}, Proposition~2.7]
\label{lemma: interpolation error}
    Let $w:\GT \to \R$ such that $w\in H^{k+1}(\Ga(t))$ for $0 \leq t \leq T$. There exists a constant $c>0$ depending on $\GT$, but independent of $h$ and $t$ such that, for $0 \leq t \leq T$,
    \begin{align*}
        \|w - I_{h}^k w \|_{L^{2}(\Ga(t))} + h \| \nbg( w - I_{h}^k w) \|_{L^{2}(\Ga(t))} \leq&\ c h^{k+1} \|w\|_{H^{k+1}(\Ga(t))} .
    \end{align*}

    We distinguish the special case of a linear surface finite element interpolation on $\Ga_h^k\t$. For $w \in H^2(\Gat)$ the linear surface finite element interpolant is denoted by $I_h^{(1)} w$, and it satisfies, with $c>0$,
    \begin{align*}
        \|w - I_h^{(1)} w \|_{L^{2}(\Ga(t))} + h \| \nbg( w - I_h^{(1)} w) \|_{L^{2}(\Ga(t))} \leq&\ c h^2 \|w\|_{H^2(\Ga(t))} .
    \end{align*}
\end{lemma}
Note that for $I_h^{(1)}$ the underlying approximating surface $\Ga_h^k\t$ is still of high order. For $k=1$, $I_h^1$ and $I_h^{(1)}$ simply coincide. The upper $k$ indices are again dropped later on.

\subsection{Velocity of lifted material points and material derivatives}
Following \cite{DziukElliott_L2} and \cite{LubichMansour_wave} we define the velocity of the lifted material points, denoted by $v_h$, and the corresponding discrete material derivative $\mat_h$.

For arbitrary $y\t = p(x\t,t) \in \Gat$, with $x\t \in \Ga_h\t$, cf.\ \eqref{eq: lift defining equation}, we have
\begin{equation}
\label{eq: definition of discrete material velocity}
    \diff y\t = v_h(y\t,t) = \pa_t p(x\t,t) + V_h(x\t,t)\cdot \nb p(x\t,t) ,
\end{equation}
hence for $y=p(x,t)$, see \cite{DziukElliott_L2},
\begin{equation*}
    v_h(y,t) = (\pr - d\wein)(x,t)V_h(x,t) - \pa_t d(x,t)\nu(x,t) - d(x,t) \pa_t \nu(x,t) .
\end{equation*}
Following Section~7.3 of \cite{LubichMansour_wave}, we note that $- \pa_t d(x,t)\nu(x,t)$ is the normal component of $v(p,t)$, and that the other terms are tangent to $\Gat$, hence
\begin{equation}
\label{eq: velocity difference is tangent}
    v-v_h \quad \textnormal{is a tangent vector} .
\end{equation}
It is also important to note that $v_h \neq V_h^l$, cf.\ \cite{DziukElliott_L2}.

The discrete material derivative of the lifted points on $\Gat$ read
\begin{equation*}
    \mat_h \vphi_h = \pa_t\vphi_h + v_h \cdot \vphi_h \qquad (\vphi_h \in (S_h\t)^l) .
\end{equation*}
The lifted basis functions also satisfy the transport property
\begin{equation*}
    \mat_h \chi_j^l = 0 , \qquad (j=1,2,\dotsc,N) .
\end{equation*}

To prove error estimates for higher order material derivatives of the Ritz map, we need high-order bounds for the error between the continuous velocity $v$ and the discrete velocity $v_h$. We generalise here Lemma~5.6 in \cite{DziukElliott_L2} and Lemma~7.3 in \cite{LubichMansour_wave} to the high-order case.
\begin{lemma}
\label{lemma: velocity estimate}
    For $\ell \geq 0$, there exits a constant $c_\ell>0$ depending on $\GT$, but independent of $t$ and $h$, such that
    \begin{align*}
        \|(\mat_h)^{(\ell)}(v - v_h)\|_{L^\infty(\Ga\t)} + h \|\nbg (\mat_h)^{(\ell)}(v - v_h)\|_{L^\infty(\Ga\t)} \leq c_\ell h^{k+1}.
    \end{align*}
\end{lemma}
\begin{proof}
We follow the steps of the original proofs from \cite{DziukElliott_L2} and \cite{LubichMansour_wave}.

(a) For $\ell=0$. Using the definition \eqref{eq: definition of discrete material velocity} and the fact $V_h = \widetilde I_h v$ we have
\begin{equation*}
 |v(p,t) - v_h(p,t)| = |\pr (v-I_h v)(p,t) + d(\wein I_h v(p,t) + \pa_t \nu)| \leq c h^{k+1} ,
\end{equation*}
where we used the interpolation estimates, Lemma~\ref{lemma: geometric est}, and the boundedness of the other terms.

For the gradient estimate we use the fact that $\nbg d = \nbg \mat_h d = 0$ and the geometric bounds of Lemma~\ref{lemma: geometric est}:
\begin{align*}
    |\nbg(v-v_h)| \leq &\ c|v-I_h v| + c|\nbg (v-I_h v)| \\
    &\ + |(\nbg d)(\wein I_h v + \pa_t \nu)| + |d(\nbg(\wein I_h v + \pa_t \nu))| \\
    \leq &\ c h^{k} .
\end{align*}

(b) For $\ell=1$, we use the transport property and again Lemma~\ref{lemma: geometric est}:
\begin{align*}
    |\mat_h(v-v_h)| \leq &\ |(\mat_h\pr) (v-I_h v)| + |\pr (\mat_hv-I_h \mat_hv)| \\
    &\ + |(\mat_hd)(\wein I_h v + \pa_t \nu)| + |d(\mat_h(\wein I_h v + \pa_t \nu))| \\
    \leq &\ c h^{k+1} .
\end{align*}

Again using $\nbg d = \nbg \mat_h d = 0$ and the geometric bounds of Lemma~\ref{lemma: geometric est}:
\begin{align*}
    |\nbg\mat_h(v-v_h)| \leq &\ c|v-I_h v| + c|\nbg (v-I_h v)| \\
    &\ + c|\mat_h (v-I_h v)| + c|\nbg (\mat_h v-I_h \mat_h v)| \\
    &\ + |(\nbg\mat_h d)(\wein I_h v + \pa_t \nu)| + |d(\nbg\mat_h(\wein I_h v + \pa_t \nu))| \\
    \leq &\ c h^{k} .
\end{align*}

(c) For $\ell >1$ the proof uses similar arguments.
\end{proof}

\subsection{Bilinear forms and their estimates}
\label{section: bilinear forms}

We use the time dependent bilinear forms defined in \cite{DziukElliott_L2}: For arbitrary $z,\vphi \in H^1(\Ga(t))$ and for their discrete analogs for $Z_h, \phi_h \in S_h(t)$
\begin{equation*}
    \begin{aligned}[c]
        m(t;z,\vphi)                &= \int_{\Ga(t)}\!\!\!\! z \vphi, \\
        a(t;z,\vphi)                &= \int_{\Ga(t)}\!\!\!\! \nbg z \cdot \nbg \vphi, \\
        g(t;v;z,\vphi)              &= \int_{\Ga(t)}\!\!\!\! (\nbg \cdot v) z\vphi, \\
        b(t;v;z,\vphi)              &= \int_{\Ga(t)}\!\!\!\! \Btensor(v) \nbg z \cdot \nbg \vphi,
    \end{aligned}
    \qquad
    \begin{aligned}[c]
        m_h(t;Z_h,\phi_h)         &= \int_{\Ga_h(t)}\!\!\!\! Z_h \phi_h, \\
        a_h(t;Z_h,\phi_h)         &= \int_{\Ga_h(t)}\!\!\!\! \nbgh Z_h \cdot \nbgh \phi_h, \\
        g_h(t;V_h;Z_h,\phi_h)     &= \int_{\Ga_h(t)}\!\!\!\! (\nbgh \cdot V_h) Z_h \phi_h, \\
        b_h(t;V_h;Z_h,\phi_h)     &= \int_{\Ga_h(t)}\!\!\!\! \Btensor_h(V_h) \nbgh Z_h \cdot \nbgh \phi_h,
    \end{aligned}
\end{equation*}
where the discrete tangential gradients are understood in a piecewise sense, and with the tensors given as
\begin{alignat*}{3}
    \Btensor(v)|_{ij} &= \delta_{ij} (\nbg \cdot v) - \big( (\nbg)_i v_j + (\nbg)_j v_i \big) , \\
    \Btensor_h(V_h)|_{ij} &= \delta_{ij} (\nbgh \cdot V_h) - \big( (\nbgh)_i (V_h)_j + (\nbgh)_j (V_h)_i \big),
\end{alignat*}
for $i,j=1,2,\dotsc,m+1$.
For more details see Lemma 2.1 in \cite{DziukElliott_L2} (and the references in the proof), or Lemma 5.2 in \cite{DziukElliott_acta}.

We will omit the time dependency of the bilinear forms if it is clear from the context.

\subsubsection{Discrete material derivative and transport properties}

The following transport equations, especially the one with discrete material derivatives on the continuous surface, are of great importance during the defect estimates later on.
\begin{lemma}[\cite{DziukElliott_L2} Lemma~4.2]
\label{lemma: transport prop}
    For any $z,\vphi, \mat_h z, \mat_h \vphi \in H^1(\Ga(t))$ we have
    \begin{align*}
        \diff m(z,\vphi) =&\ m(\mat_h z,\vphi) + m(z,\mat_h \vphi) + g(v_h;z,\vphi) , \\
        \diff a(z,\vphi) =&\ a(\mat_h z,\vphi) + a(z,\mat_h \vphi) + b(v_h;z,\vphi) .
    \end{align*}
    The same formulas hold when $\mat_h$ and $v_h$ are replaced with $\mat$ and $v$, respectively.

    Similarly, in the discrete case, for arbitrary $z_h, \phi_h, \mat_hz_h, \mat_h\phi_h \in S_h(t)$ we have
    \begin{align*}
        \diff m_h(z_h,\phi_h) =&\ m_h(\mat_h z_h,\phi_h) + m_h(z_h,\mat_h \phi_h) + g_h(V_h;z_h,\phi_h) , \\
        \diff a_h(z_h,\phi_h) =&\ a_h(\mat_h z_h,\phi_h) + a_h(z_h,\mat_h \phi_h) + b_h(V_h;z_h,\phi_h) .
    \end{align*}
\end{lemma}

\subsubsection{Geometric perturbation errors}

The following estimates are the most important technical results of this paper. Later on, they will play a crucial role in the defect estimates. We note here that these results extend the first order ESFEM theory of \cite{DziukElliott_L2} (for the first three inequalities) and \cite{LubichMansour_wave} (for the last inequality) to the higher order ESFEM case.

The high-order version of the inequalities for time independent bilinear forms $a(\cdot,\cdot)$ and $m(\cdot,\cdot)$ were shown in Lemma~6.2 in \cite{ElliottRanner}. The following results generalises these for the time dependent case.
\begin{lemma}
\label{lemma: geometric perturbation errors}
    For any $Z_h,\phi_h \in S_h^k(t)$, and for their lifts $Z_h^l,\phi_h^l \in H^1(\Ga(t))$, we have the following bounds
    \begin{align*}
        \big| m( Z_h^l,\phi_h^l) - m_h( Z_h,\phi_h) \big| \leq&\ c h^{k+1} \|Z_h^l\|_{L^2(\Ga(t))} \|\phi_h^l\|_{L^2(\Ga(t))} , \\
        \big| a( Z_h^l,\phi_h^l) - a_h( Z_h,\phi_h) \big| \leq&\ c h^{k+1} \|\nbg Z_h^l\|_{L^2(\Ga(t))} \|\nbg \phi_h^l\|_{L^2(\Ga(t))} , \\
        \big| g(v_h; Z_h^l,\phi_h^l) - g_h(V_h; Z_h,\phi_h) \big| \leq&\ c h^{k+1} \|Z_h^l\|_{L^2(\Ga(t))} \|\phi_h^l\|_{L^2(\Ga(t))} , \\
        \big| b(v_h; Z_h^l,\phi_h^l) - b_h(V_h; Z_h,\phi_h) \big| \leq&\ c h^{k+1} \|\nbg  Z_h^l\|_{L^2(\Ga(t))} \|\nbg \phi_h^l\|_{L^2(\Ga(t))} ,
    \end{align*}
    where the constants $c>0$ are independent of $h$ and $t$, but depend on $\GT$.
\end{lemma}

\begin{proof}
The proof of the first two estimates is a high-order generalization of Lemma~5.5 in \cite{DziukElliott_L2}, 
while the proof of the last two estimates is a high-order extension of Lemma~7.5 in \cite{LubichMansour_wave}. Their proofs follow these references. Both parts are using the geometric estimates showed in Lemma~\ref{lemma: geometric est}.

To show the first inequality we estimate as
\begin{align*}
    \big| m( Z_h^l,\phi_h^l) - m_h( Z_h,\phi_h) \big| =&\ \Big| \int_{\Ga(t)} \!\!\!\! Z_h^l \phi_h^l \d A - \int_{\Ga_h(t)} \!\!\!\! Z_h \phi_h \d A_h \Big| \\
    = &\ \Big| \int_{\Ga(t)} (1 - \delta_h\inv) Z_h^l \phi_h^l \d A \Big| \\
    \leq &\ c h^{k+1} \|Z_h^l\|_{L^2(\Ga(t))} \|\phi_h^l\|_{L^2(\Ga(t))} .
\end{align*}

For the second inequality, similarly we have
\begin{align*}
    \big| a( Z_h^l,\phi_h^l) - a_h( Z_h,\phi_h) \big| =&\  \Big| \int_{\Ga(t)} \!\!\!\!\! \nbg Z_h^l \cdot \nbg \phi_h^l \d A - \int_{\Ga_h(t)} \!\!\!\!\!  \nbgh Z_h \cdot \nbgh \phi_h \d A_h \Big| \\
   = &\ \Big| \int_{\Ga(t)} (\Id - \delta_h Q_h) \nbg Z_h^l \cdot \nbg \phi_h^l \d A \Big| \\
    \leq &\ c h^{k+1} \|\nbg Z_h^l\|_{L^2(\Ga(t))} \|\nbg \phi_h^l\|_{L^2(\Ga(t))} .
\end{align*}

For the third estimate we start by taking the time derivative of the equality $m(Z_h^l,\phi_h^l) = m_h(Z_h,\phi_h \delta_h)$, using the first transport property from Lemma~\ref{lemma: transport prop} to obtain
\begin{align*}
    \diff m(Z_h^l,\phi_h^l) =&\ m(\mat_h Z_h^l,\phi_h^l) + m(Z_h^l,\mat_h \phi_h^l) \\
    &\ + g(v_h;Z_h^l,\phi_h^l)\\
    =\diff m_h(Z_h,\phi_h \delta_h) =&\ m_h(\mat_h Z_h,\phi_h \delta_h) + m_h(Z_h,(\mat_h \phi_h)\delta_h) \\
    &\ + g_h(V_h;Z_h,\phi_h\delta_h) + m_h(Z_h,(\mat_h \delta_h)\phi_h) .
\end{align*}
Hence, using $\mat_h w_h^l = (\mat_h w_h)^l$,
\begin{align*}
    g(v_h;Z_h^l,\phi_h^l) - g_h(V_h;Z_h,\phi_h\delta_h) =&\ m((\mat_h Z_h)^l,\phi_h^l) - m_h(\mat_h Z_h,\phi_h \delta_h) \\
    &\ + m(Z_h^l,(\mat_h \phi_h)^l) - m_h(Z_h,(\mat_h \phi_h)\delta_h) \\
    &\ + m_h(Z_h,(\mat_h \delta_h)\phi_h) \\
    =&\ m_h(Z_h,(\mat_h \delta_h)\phi_h) .
\end{align*}
Hence, together with Lemma \ref{lemma: geometric est}, the bound
\begin{align*}
    &\ |g(v_h;Z_h^l,\phi_h^l) - g_h(V_h;Z_h,\phi_h)| \\
    \leq &\ |g_h(V_h;Z_h,\phi_h(\delta_h-1))| + |m_h(Z_h,(\mat_h \delta_h)\phi_h)| \\
    \leq &\ c \Big( \|\mat_h \delta_h\|_{L^\infty(\Ga_h(t))} + \|1-\delta_h\|_{L^\infty(\Ga_h(t))} \Big) \|Z_h^l\|_{L^2(\Ga(t))} \|\phi_h^l\|_{L^2(\Ga(t))}
\end{align*}
finishes the proof of the third estimate.

For the last inequality we take a similar approach, using the second transport property from Lemma~\ref{lemma: transport prop} and the relation \eqref{eq: gradient transformation} to obtain
\begin{align*}
    &\ \diff a_h(Z_h,\phi_h) = a_h(\mat_h Z_h,\phi_h) + a_h(Z_h,\mat_h \phi_h) + b_h(V_h;Z_h,\phi_h)\\
    =&\ \diff \int_{\Gat} Q_h^l\nbg Z_h^l \cdot \nbg \phi_h^l = \int_{\Gat} Q_h^l\nbg \mat_h Z_h^l \cdot \nbg \phi_h^l + \int_{\Gat} Q_h^l\nbg Z_h^l \cdot \nbg \mat_h \phi_h^l \\
    &\ \hphantom{\diff \int_{\Gat} Q_h^l\nbg Z_h^l \cdot \nbg \phi_h^l = } + \int_{\Gat} (\mat_hQ_h^l) \nbg Z_h^l \cdot \nbg \phi_h^l + \int_{\Gat} \Btensor(v_h) Q_h^l\nbg Z_h^l \cdot \nbg \phi_h^l .
\end{align*}
Again, using $\mat_h w_h^l = (\mat_h w_h)^l$, together with \eqref{eq: gradient transformation} and the geometric estimates of Lemma~\ref{lemma: geometric est} we obtain
\begin{align*}
    |b_h(V_h;Z_h,\phi_h) - b(v_h;Z_h^l,\phi_h^l)| = &\ \Big|\int_{\Gat} (\mat_hQ_h^l) \nbg Z_h^l \cdot \nbg \phi_h^l \\
    &\ + \int_{\Gat} \Btensor(v_h) \big(Q_h^l - \Id \big)\nbg Z_h^l \cdot \nbg \phi_h^l \Big| \\
    \leq &\ ch^{k+1} \|\nbg  Z_h^l\|_{L^2(\Ga(t))} \|\nbg \phi_h^l\|_{L^2(\Ga(t))},
\end{align*}
completing the proof.
\end{proof}

\section{Generalized Ritz map and higher order error bounds}
\label{section: Ritz}

We recall the generalised Ritz map for evolving surface PDEs from \cite{LubichMansour_wave}.

\begin{definition}[Ritz map]
\label{def: Ritz}
    For any given $z\in H^{1}(\Gat)$ there is a
    unique $\Pt z\in S_h^k\t$ \st\ for all $\phi_h\in S_h^k\t$, with the corresponding lift
    $\vphi_h=\phi_h^l$, we have
    \begin{equation}\label{eq: Ritz definition}
        a_h^{\ast}(\Pt z,\phi_h) = a^\ast(z,\vphi_h),
    \end{equation}
    where we let $a^{\ast}=a+m$ and $a_h^{\ast}=a_h+m_h$, to make the forms $a$ and $a_h$ positive
    definite. Then $\P z \in (S_h^k\t)^l$ is defined as the lift of $\Pt z$, i.e.\ $\P z = (\Pt z)^l$.
\end{definition}
We note here that originally in Definition~8.1 in \cite{LubichMansour_wave} an extra term appeared involving $\mat z$ and the surface velocity, which is not needed for the parabolic case. The Ritz map above is still well-defined.

Galerkin orthogonality does not hold in this case, just up to a small defect.
\begin{lemma}[Galerkin orthogonality up to a small defect]
\label{lemma: galerkin orthogonality}
    For any $z\in H^{1}(\Gat)$ and $\varphi_{h}\in (S_h^k\t)^l$
    \begin{equation}\label{eq: galerkin orthogonality}
        \abs{\aast(z-\P z,\varphi_{h}) } \leq c h^{k+1} \| \P z \|_{H^1(\Gat)} \|\vphi_h\|_{H^1(\Gat)},
    \end{equation}
    where $c$ is independent of $\xi$,\ $h$ and $t$.
\end{lemma}
\begin{proof}
Using the definition of the Ritz map and  Lemma~\ref{lemma: geometric perturbation errors} we estimate by
\begin{align*}
    \abs{\aast(z-\P z, \vphi_h)} =&\ \abs{\aast_h(\Pt z,  \phi_h) - \aast(\P z, \vphi_h)}
    \leq   c h^{k+1} \| \P z \|_{H^1(\Gat)} \|\vphi_h\|_{H^1(\Gat)} .
\end{align*}
\end{proof}

\subsection{Errors in the Ritz map}
Now we prove higher order error estimates for the Ritz map \eqref{eq: Ritz definition} and also for its material derivatives, the analogous results for the first order ESFEM case can be found in Section~6 of \cite{DziukElliott_L2}, or Section~7 of \cite{diss_Mansour}.

\subsubsection{Error bounds for the Ritz map}

\begin{theorem}
\label{theorem: Ritz error}
    Let $z:\GT \to \R$ with $z\in H^{k+1}(\Gat)$ for every $0 \leq t \leq T$. Then the error in the Ritz map satisfies the bound, for $0 \leq t \leq T$ and $h \leq h_0$ with sufficiently small $h_0$,
    \begin{equation*}
        \|z-\P z\|_{L^2(\Gat)} + h \|z-\P z\|_{H^1(\Gat)} \leq c h^{k+1} \|z\|_{H^{k+1}(\Gat)} ,
    \end{equation*}
    where the constant $c>0$ is independent of $h$ and $t$.
\end{theorem}
\begin{proof}
To ease the presentation we suppress all time arguments $t$ appearing in the norms within the proof (except some special occasions).

(a) We first prove the gradient estimate.
Starting by the definition of the $H^1(\Ga\t)$ norm, then using the estimate \eqref{eq: galerkin orthogonality} we have
\begin{align*}
    \|z-\P z\|_{H^1(\Ga)}^2 = &\ \aast(z-\P z, z-\P z) \\
    =&\ \aast(z-\P z, z-I_h z) + \aast(z-\P z, I_h z-\P z) \\
    \leq&\  \|z-\P z\|_{H^1(\Ga)} \|z-I_h z\|_{H^1(\Ga)}  + c h^{k+1} \| \P z \|_{H^1(\Ga)} \|I_h z-\P z\|_{H^1(\Ga)} \\
    \leq&\  c h^k \|z-\P z\|_{H^1(\Ga)} \|z\|_{H^{k+1}(\Ga)} \\
    &\ + c h^{k+1} \Big( \! 2\|z-\P z\|_{H^1(\Ga)}^2 \! + \! \|z\|_{H^1(\Ga)}^2 \! + \! ch^{2k}\|z\|_{H^{k+1}(\Ga)}^2 \! \Big),
\end{align*}
using interpolation error estimate, and for the second term we used the estimate
\begin{align*}
    &\ \| \P z \|_{H^1(\Ga)} \|I_h z-\P z\|_{H^1(\Ga)} \\
    \leq &\ \Big(\| \P z -z\|_{H^1(\Ga)} + \|z\|_{H^1(\Ga)} \Big) \Big(\|I_h z-z\|_{H^1(\Ga)} + \|z-\P z\|_{H^1(\Ga)}\Big)\\
    \leq &\ 2\|z-\P z\|_{H^1(\Ga)}^2 + \|z\|_{H^{k+1}(\Ga)}^2 + ch^{2k}\|z\|_{H^{k+1}(\Ga)}^2.
\end{align*}
Now using Young's inequality, and for a sufficiently small (but $t$ independent) $h\leq h_0$ we have the gradient estimate
\begin{equation*}
    \|z-\P z\|_{H^1(\Ga\t)} \leq  c h^k \|z\|_{H^{k+1}(\Ga\t)}.
\end{equation*}

(b) The $L^2$-estimate follows from the Aubin--Nitsche trick. Let us consider the problem
\begin{equation*}
    -\laplace_{\Ga\t} w + w = z-\P z \qquad \textrm{on}\quad  \Ga\t,
\end{equation*}
then by the usual elliptic theory (see, e.g.\ Section 3.1 of \cite{DziukElliott_acta}, or \cite{Aubinbook})
yields: the solution $w\in H^2(\Ga\t)$ satisfies the bound, with a $c>0$ independent of $t$,
\begin{equation*}
    \|w\|_{H^2(\Ga\t)} \leq c \|z-\P z\|_{L^2(\Ga\t)} .
\end{equation*}

By testing the elliptic weak problem with $z-\P z$, using \eqref{eq: galerkin orthogonality} again, and using the linear finite element interpolation $I_h^{(1)}$ on $\Ga_h^k\t$, we obtain
\begin{align*}
    \|z-\P z\|^2_{L^2(\Ga)} =&\ \aast(z-\P z,w) \\
    =&\ \aast(z-\P z,w - I_h^{(1)} w) + \aast(z-\P z,I_h^{(1)} w) \\
    \leq&\  \|z-\P z\|_{H^1(\Ga)} \|w - I_h^{(1)} w\|_{H^1(\Ga)}  + ch^{k+1} \| \P z \|_{H^1(\Ga)} \|I_h^{(1)} w\|_{H^1(\Ga)} \\
    \leq &\  ch^k\|z\|_{H^{k+1}(\Ga)} ch\|w\|_{H^2(\Ga)}  + ch^{k+1} \| \P z \|_{H^1(\Ga)} \|I_h^{(1)} w\|_{H^1(\Ga)} .
\end{align*}
Where for the second term we now used
\begin{align*}
    \|\P z\|_{H^1(\Ga)} \|I_h^{(1)} w\|_{H^1(\Ga)}
    \leq &\ (\|z-\P z\|_{H^1(\Ga)} + \|z\|_{H^1(\Ga)} )(\|w-I_h^{(1)} w\|_{H^1(\Ga)} + \|w\|_{H^1(\Ga)}) \\
    \leq &\ (1+ch^k)\|z\|_{H^{k+1}(\Ga)}(1+ch)\|w\|_{H^2(\Ga)}.
\end{align*}

Then the combination of the gradient estimate for the Ritz map and the interpolation error yields
\begin{equation*}
    \|z-\P z\|_{L^2(\Ga\t)} \frac{1}{c} \|w\|_{H^2(\Ga\t)} \leq \|z-\P z\|_{L^2(\Ga\t)}^{2} \leq  c h^{k+1} \|z\|_{H^{k+1}(\Ga\t)} \|w\|_{H^2(\Ga\t)},
\end{equation*}
which completes the proof.
\end{proof}

\subsubsection{Error bounds for the material derivatives of the Ritz map}

Since, in general $\mat_h \P z = \P \mat_h z$ does not hold, we need the following result.

\begin{theorem}
\label{theorem: Ritz mat error}
    The error in the material derivatives of the Ritz map, for any $\ell\in\N$, satisfies the following bounds, for $0 \leq t \leq T$ and $h \leq h_0$ with sufficiently small $h_0>0$
    \begin{align*}
        \|(\mat_h)^{(\ell)}(z-\P z)\|_{L^2(\Gat)} + h &\|\nbg(\mat_h)^{(\ell)}(z-\P z)\|_{L^2(\Gat)}  \leq  c_{\ell} h^{k+1} \sum_{j=0}^{\ell} \|(\mat)^{(j)}z\|_{H^{k+1}(\Gat)} ,
    \end{align*}
    where the constant $c_\ell>0$ is independent of $h$ and $t$.
\end{theorem}

\begin{proof}
The proof is a modification of Theorem 7.3 in \cite{diss_Mansour}.
Again, to ease the presentation we suppress the $t$ argument of the surfaces norms (except some special occasions).

\medskip
For $\ell=1$: (a) We start by taking the time derivative of the definition of the Ritz map \eqref{eq: Ritz definition}, use the transport properties Lemma~\ref{lemma: transport prop}, and apply the definition of the Ritz map once more, we arrive at
\begin{align*}
    \aast(\mat_h z,\vphi_h) =& -b(v_h;z,\vphi_h) - g(v_h;z,\vphi_h) \\
     & +\aast_h(\mat_h \Pt z,\phi_h) +b_h(V_h;\Pt z,\phi_h) +g_h(V_h;\Pt z,\phi_h).
\end{align*}
Then we obtain
\begin{equation}
\label{eq: mat error ritz - main equation}
    \begin{aligned}
        \aast(\mat_h z-\mat_h \P z ,\vphi_h) =& -b(v_h;z - \P z,\vphi_h) - g(v_h;z - \P z,\vphi_h) \\
         & +F_1(\vphi_h),
    \end{aligned}
\end{equation}
where
\begin{align*}
     F_1(\vphi_h) =&\ \big(\aast_h(\mat_h \Pt z,\phi_h) - \aast(\mat_h \P z ,\vphi_h)\big) \\
                  &\ + \big(b_h(V_h;\Pt z,\phi_h) - b(v_h;\P z,\vphi_h)\big) \\
                  &\ + \big(g_h(V_h;\Pt z,\phi_h) - g(v_h;\P z,\vphi_h)\big) .
\end{align*}
Using the geometric estimates of Lemma \ref{lemma: geometric perturbation errors}, $F_1$ can be estimated as
\begin{equation}
\label{eq: F_1 bound}
    \big|F_1(\vphi_h)\big| \leq ch^{k+1} \Big( \|\mat_h\P z\|_{H^1(\Ga\t)} + \|\P z\|_{H^1(\Ga\t)}\Big) \|\vphi_h\|_{H^1(\Ga\t)}.
\end{equation}
The velocity error estimate Lemma~\ref{lemma: velocity estimate} yields
\begin{equation*}
    \|\mat_h z\|_{H^1(\Ga)} \leq \|\mat z\|_{H^1(\Ga)} + ch^k \|z\|_{H^2(\Ga)} .
\end{equation*}
Then using $\mat_h \P z$ as a test function in \eqref{eq: mat error ritz - main equation}, and using the error estimates of the Ritz map, together with the estimates above, with $h\leq h_0$, we have
\begin{align*}
    \|\mat_h \P z\|_{H^1(\Ga)}^2 =&\ \aast(\mat_h \P z ,\mat_h \P z) \\
    =& b(v_h;z - \P z,\mat_h \P z) + g(v_h;z - \P z,\mat_h \P z)  + \aast(\mat_h z ,\mat_h \P z) - F_1(\mat_h \P z) \\
    \leq &\ ch^{k}\|z\|_{H^{k+1}(\Ga)}\|\mat_h\P z\|_{H^1(\Ga)} + \|\mat_h z\|_{H^1(\Ga)}\|\mat_h \P z\|_{H^1(\Ga)} \\
    &\ + ch^{k+1} \big( \|\mat_h\P z\|_{H^1(\Ga)} + \|\P z\|_{H^1(\Ga)}\big) \|\mat_h \P z\|_{H^1(\Ga)} \\
    \leq &\ ch^{k}\|z\|_{H^{k+1}(\Ga)}\|\mat_h\P z\|_{H^1(\Ga)}
     + (\|\mat z\|_{H^1(\Ga)} + ch^k \|z\|_{H^2(\Ga)})\|\mat_h \P z\|_{H^1(\Ga)} \\
    &\ + ch^{k+1} \big( \|\mat_h\P z\|_{H^1(\Ga)} + \|z - \P z\|_{H^1(\Ga)} + \|z\|_{H^1(\Ga)}\big) \|\mat_h \P z\|_{H^1(\Ga)} \\
    \leq &\ \Big(ch^{k}\|z\|_{H^{k+1}(\Ga)} + \|\mat z\|_{H^1(\Ga)}\Big) \|\mat_h\P z\|_{H^1(\Ga)}
     + ch^{k+1} \|\mat_h\P z\|_{H^1(\Ga)}^2,
\end{align*}
absorbtion using an $h\leq h_0$, with a sufficiently small $h_0>0$, and dividing through yields
\begin{equation*}
    \|\mat_h \P z\|_{H^1(\Ga)} \leq  \|\mat z\|_{H^1(\Ga)} +  ch^{k}\|z\|_{H^{k+1}(\Ga)} .
\end{equation*}

Combining all the previous estimates and using Young's inequality, Cauchy--Schwarz inequality, and Theorem~\ref{theorem: Ritz error}, for sufficiently small $h\leq h_0$, we obtain
\begin{align*}
     \aast(\mat_h z-\mat_h \P z ,\vphi_h) \leq&\ c \|z - \P z\|_{H^1(\Ga)} \|\vphi_h\|_{H^1(\Ga)} \\
     & + ch^{k+1} \big( \|\mat_h\P z\|_{H^1(\Ga)} + \|\P z\|_{H^1(\Ga)}\big) \|\vphi_h\|_{H^1(\Ga)} \\
     \leq&\ c h^k \|z\|_{H^{k+1}(\Ga)} \|\vphi_h\|_{H^1(\Ga)} \\
     & + ch^{k+1} \big( \|\mat z\|_{H^1(\Ga)} + (1+ch^k)\|z\|_{H^{k+1}(\Ga)}\big) \|\vphi_h\|_{H^1(\Ga)} \\
     \leq&\ c h^k \Big( \|z\|_{H^{k+1}(\Ga)} + h\|\mat z\|_{H^1(\Ga)} \Big) \|\vphi_h\|_{H^1(\Ga)}.
\end{align*}

Then, similarly to the previous proof we have
\begin{align*}
    \|\mat_hz-\mat_h\P z\|_{H^1(\Ga)}^2
    \leq&\ \aast(\mat_h z-\mat_h \P z, \mat_h z-\mat_h \P z) \\
    =&\ \aast(\mat_h z - \mat_h \P z, \mat_h z-I_h \mat z) + \aast(\mat_h z-\mat_h \P z, I_h \mat z-\mat_h \P z) \\
    \leq&\  \|\mat_h z-\mat_h \P z\|_{H^1(\Ga)} \|\mat_h z-I_h \mat z\|_{H^1(\Ga)} \\
    &\ +  c h^k \Big( \|z\|_{H^{k+1}(\Ga)} + h\|\mat z\|_{H^1(\Ga)} \Big) \|I_h \mat z-\mat_h \P z\|_{H^1(\Ga)}.
\end{align*}
Finally the interpolation estimates, Young's inequality, absorption using a sufficiently small $h\leq h_0$, yields the gradient estimate.

(b) The $L^2$-estimate again follows from the Aubin--Nitsche trick. Let us now consider the problem
\begin{equation*}
    -\laplace_{\Ga\t}w + w = \mat_h z-\mat_h \P z \qquad \textrm{on}\quad  \Ga\t,
\end{equation*}
together with the usual elliptic estimate, for the solution $w\in H^2(\Ga\t)$
\begin{equation*}
    \|w\|_{H^2(\Ga\t)} \leq c\|\mat_h z-\mat_h \P z\|_{L^2(\Ga\t)},
\end{equation*}
again, $c$ is independent of $t$ and $h$.

Following the proof of Theorem~6.2 in \cite{DziukElliott_L2}, let us first bound
\begin{align*}
    -b(v_h;z - \P z,I_h^{(1)} w) = &\ b(v_h;z - \P z,w-I_h^{(1)} w) - b(v_h;z - \P z,w) \\
    \leq &\ ch^{k} \|z\|_{H^{k+1}(\Ga)} \ ch \|w\|_{H^2(\Ga)} - b(v_h;z - \P z,w) \\
    =&\ ch^{k+1} \|z\|_{H^{k+1}(\Ga)} \|w\|_{H^2(\Ga)} + b(v;z - \P z,w) \\
    &\ + b(v;z - \P z,w) - b(v_h;z - \P z,w),
\end{align*}
where again $I_h^{(1)}$ denotes the linear finite element interpolation operator on $\Ga_h^k\t$.

The pair in the last line can be estimated, using Lemma~\ref{lemma: velocity estimate}, by
\begin{align*}
    b(v;z - \P z,w) - b(v_h;z - \P z,w)
    \leq &\ \int_{\Gat} |\Btensor(v)-\Btensor(v_h)| |\nbg(z - \P z)| |\nbg w| \\
    \leq &\ ch^{2k} \|z\|_{H^{k+1}(\Ga)} \|w\|_{H^1(\Ga)} .
\end{align*}
Finally, for the remaining term the proof of Theorem~6.2 in \cite{DziukElliott_L2} yields
\begin{equation*}
    b(v;z - \P z,w) \geq -c\|z - \P z\|_{L^2(\Ga)} \|w\|_{H^2(\Ga)} \geq -c h^{k+1}\|z\|_{H^{k+1}(\Ga)}\|w\|_{H^2(\Ga)} .
\end{equation*}
For the other bilinear form in \eqref{eq: mat error ritz - main equation} we have
\begin{equation*}
    g(v_h;z - \P z,\vphi_h) \leq ch^{k+1}\|z\|_{H^{k+1}(\Ga)}\|w\|_{H^1(\Ga)} .
\end{equation*}
The combination of all these estimates with \eqref{eq: F_1 bound}, yields
\begin{equation*}
    \aast(\mat_h z -\mat_h\P z,I_h w) \leq ch^{k+1}\|z\|_{H^{k+1}(\Ga)}\|w\|_{H^2(\Ga)} .
\end{equation*}

By testing the above elliptic weak problem with $z-\P z$, and using the above bound and the gradient estimate from (a), we obtain
\begin{align*}
    \|\mat_h z-\mat_h \P z\|_{L^2(\Ga)}^2 =&\ \aast(\mat_h z-\mat_h \P z,w) \\
    =&\ \aast(\mat_h z-\mat_h \P z,w-I_h^{(1)} w) + \aast(\mat_h z-\mat_h \P z,I_h^{(1)} w) \\
    \leq&\ ch^k \Big( \|z\|_{H^{k+1}(\Ga)} + h\|\mat z\|_{H^1(\Ga)} \Big) \ ch \|w\|_{H^2(\Ga\t)}
     + ch^{k+1}\|z\|_{H^{k+1}(\Ga)}\|w\|_{H^2(\Ga)} \\
    \leq &\ ch^{k+1} \Big( \|z\|_{H^{k+1}(\Ga)} + h\|\mat z\|_{H^1(\Ga)} \Big) \|w\|_{H^2(\Ga)} .
\end{align*}

For $\ell>1$ the proof is analogous.
\end{proof}

\section{Error bounds for the semi- and full discretisation}
\label{section: error bounds}

\subsection{Convergence proof for the semidiscretisation}
By combining the error estimates in the Ritz map and in its material derivatives and the geometric results of Section~\ref{section: geometric estimates} we prove convergence of the high-order ESFEM semidiscretisation.

\smallskip
\begin{proof}[Proof of Theorem~\ref{theorem: semidiscrete convergence}]
    The result is simply shown by repeating the arguments of Section~7 in \cite{DziukElliott_L2} for our setting, but using the high-order versions of all results: geometric estimates Lemma~\ref{lemma: geometric est}, perturbation estimates of bilinear forms Lemma \ref{lemma: geometric perturbation errors} and Ritz map error estimates Lemma~\ref{theorem: Ritz error} and \ref{theorem: Ritz mat error}.
\end{proof}

\subsection{Convergence proof for the full discretisation}

\subsubsection{Bound of the semidiscrete residual}

We follow the approach of Section~8.1 of \cite{diss_Mansour} and Section~5 of \cite{LubichMansourVenkataraman_bdsurf}. By defining the ESFEM residual $R_h(\cdot,t) = \sum_{j=1}^N r_j\t \chi_j(\cdot,t)\in S_h^k\t$ as
\begin{equation}
\label{eq: residual}
    \int_{\Ga_h^k\t}\!\!\!\!  R_h \phi_h = \diff \int_{\Ga_h^k\t}\!\!\!\!  \Pt u \phi_h +\int_{\Ga_h^k\t}\! \nbgh (\Pt u) \cdot \nbgh \phi_h - \int_{\Ga_h^k\t}\!\!\!\! (\Pt u)  \mat_h \phi_h,
\end{equation}
where $\phi_h\in S_h^k\t$, and $\Pt u(\cdot,t)$ is the Ritz map of the smooth solution $u$.


We now show the optimal order $H_h\inv$-norm estimate of the residual $R_h$.
\begin{theorem}
\label{theorem: res bound}
    Let the solution $u$ of the parabolic problem be sufficiently smooth. Then there exist $C>0$ and $h_0>0$, \st\ for all $h\leq h_0$ and $t\in[0,T]$, the finite element residual $R_h$ of the Ritz map is bounded as 
    \begin{equation*}
        \|R_h\|_{H_h\inv(\Ga_h\t)} \leq C h^{k+1}.
    \end{equation*}
\end{theorem}

\begin{proof}
(a) We start by applying the discrete transport property to the residual equation \eqref{eq: residual}
\begin{align*}
    m_h(R_h,\phi_h) =&\ \diff m_h(\Pt u,\phi_h) + a_h(\Pt u,\phi_h) - m_h(\Pt u,\mat_h\phi_h) \\
    =&\ m_h(\mat_h\Pt u,\phi_h) + a_h(\Pt u,\phi_h) + g_h(V_h;\Pt u,\phi_h).
\end{align*}

(b) We continue by the transport property with discrete material derivatives from Lemma \ref{lemma: transport prop}, but for the weak form, with $\vphi:=\vphi_h=(\phi_h)^l$
\begin{align*}
    0 =&\ \diff m(u,\vphi_h) + a(u,\vphi_h) - m(u,\mat\vphi_{h}) \\
    =&\ m(\mat_h u,\vphi_h) + a(u,\vphi_h) + g(v_h;u,\vphi_h) +
    m(u,\mat_{h}\varphi_{h} - \mat \varphi_{h}).
\end{align*}

(c) Subtraction of the two equations, using the definition of the Ritz map \eqref{eq: Ritz definition},
and using that
\begin{equation*}
    \mat_{h}\varphi_{h} - \mat \varphi_{h} = (v_{h}-v) \cdot \nabla_{\Gamma}\varphi_{h}
\end{equation*}
holds, we obtain
\begin{align*}
    m_h(R_h,\phi_h) =&\ m_h(\mat_h\Pt u,\phi_h) - m(\mat_h u,\vphi_h) \\
                    &\ + g_h(V_h;\Pt u,\phi_h) - g(v_h;u,\vphi_h) \\
                    &\ + m(u,\vphi_h) - m_h(\Pt u, \phi_h) \\
                    &\ + m(u, (v_{h}-v)  \cdot \nabla_{\Gamma}\varphi_{h}).
\end{align*}
All the pairs can be easily estimated separately as $c h^{k+1} \|\vphi_h\|_{H^1(\Gat)}$ by combining the geometric perturbation estimates of Lemma~\ref{lemma: geometric perturbation errors}, the velocity estimate of Lemma~\ref{lemma: velocity estimate}, and the error estimates of the Ritz map from Theorem~\ref{theorem: Ritz error} and \ref{theorem: Ritz mat error}. The proof is finished using the definition of $H_h\inv$-norm and the equivalence of norms Lemma~\ref{lemma: equivalence of discrete norms}.
\end{proof}


\subsubsection{Proof of Theorem~\ref{theorem: ESFEM BDF error}}
Using the error estimate for the BDF methods Theorem~\ref{theorem: BDF error estimates} and using the bounds for the semidiscrete residual Theorem~\ref{theorem: res bound} we give here a proof for the fully discrete error estimates of Theorem~\ref{theorem: ESFEM BDF error}.

\begin{proof}[Proof of Theorem~\ref{theorem: ESFEM BDF error}]
    The global error is decomposed into two parts
    \begin{equation*}
        u_h^n - u(\cdot,t_n) = \big(u_h^n - (\P u)(\cdot,t_n)\big) + \big((\P u)(\cdot,t_n) - u(\cdot,t_n)\big),
    \end{equation*}
    and the terms are estimated by results from above.

    The first one is estimated by results for BDF methods Theorem~\ref{theorem: BDF error estimates}  together with the residual bound Theorem~\ref{theorem: res bound} and by the Ritz map error estimates Theorem~\ref{theorem: Ritz error} and \ref{theorem: Ritz mat error}.

    The second term is only estimated by the error estimates for the Ritz map and for its material derivatives.
\end{proof}
\begin{remark}
    In Remark~\ref{remark: Runge--Kutta extension} we noted that the analogous fully discrete results can be shown for algebraically stable implicit Runge--Kutta methods. To be more precise: for the Runge--Kutta analog, instead of Theorem~\ref{theorem: BDF error estimates} one has to use the error bounds of Theorem~5.1 in \cite{DziukLubichMansour_rksurf}, but otherwise the proof remains the same.
\end{remark}

\section{Implementation}
\label{section: implementation}

The implementation of the high-order ESFEM code follows the typical way of finite element mass matrix, stiffness matrix and load vector assembly, mixed with techniques from isoparametric FEM theory. This was also used for linear ESFEM, see Section~7.2 of \cite{DziukElliott_ESFEM}.

Similarly to the linear case, a curved element $E_h$ of the $k$-order interpolation surface $\Ga_h\t$ is paramet\-ri\-sed over the reference triangle $E_0$, chosen to be the unit triangle. Then the polynomial map of degree $k$ between $E_h$ and $E_0$ is used to compute the local matrices. All the computations are done on the reference element, using the Dunavant quadrature rule, see \cite{Dunavant}. 
Then the local values are summed up to their correct places in the global matrices.

In a typical case the surface is evolved by solving a series of ODEs, hence only the initial mesh is created based on $\Ga(0)$. Naturally, the problem of velocity based grid distortion is still present. Possible overcomes are methods using the DeTurck trick see \cite{ElliottFritz}, or using ALE finite elements see \cite{ElliottVenkataraman_ALE}, \cite{KovacsPower_ALEdiscrete}.

\section{Numerical experiments}
\label{section: numerics}

We performed various numerical experiments with quadratic approximation of the surface $\Gat$ and using quadratic ESFEM to illustrate our theoretical results.

\subsection{Example 1: Parabolic problem on a stationary surface}

Let us shortly report on numerical experiments for parabolic problems on a stationary surface, as a benchmark problem. Let $\Ga\subset\R^3$ be the unit sphere, and let us consider the parabolic surface PDE
\begin{equation*}
    \pa_t u - \laplace_\Ga u = f,
\end{equation*}
with given initial value and inhomogeneity $f$ chosen such, that the solution is $u(x,t)=e^{-6t}x_1x_2$.

Let $(\mathcal{T}_k)_{k=1,2,\dotsc,n}$ and $(\tau_k)_{k=1,2,\dotsc,n}$ be a series of meshes and time steps, respectively, such that $2 h_k = h_{k-1}$ and $2 \tau_k = \tau_{k-1}$, with $h_1=\sqrt2$ and $\tau_1=0.2$. For each mesh $\mathcal{T}_k$ with corresponding stepsize $\tau_k$ we numerically solve the surface PDE using second order ESFEM combined with the third order BDF methods. Then, by $e_k$ we denote the error corresponding to the mesh $\mathcal{T}_k$ and step size $\tau_k$. We then compute the errors between the lifted numerical solution and the exact solution using the following norm and seminorm:
\begin{alignat*}{3}
    L^\infty(L^2):&\qquad \max_{1\leq n \leq N}\|u_h^n - u(\cdot,t_n)\|_{L^2(\Ga(t_n))},\\
    L^2(H^1):&\qquad  \Big(\tau \sum_{n=1}^{N} \|\nb_{\Ga(t_n)}\big(u_h^n - u(\cdot,t_n)\big)\|_{L^2(\Ga(t_n))}^2\Big)^{1/2}.
\end{alignat*}
Using the above norms, the experimental order of convergence rates (EOCs) are computed  by
\begin{equation*}
    \textnormal{EOC}_{k}=\frac{\ln(e_{k}/e_{k-1})}{\ln(2)}, \qquad (k=2,3, \dotsc,n).
\end{equation*}

In Table~\ref{table: stationary EOCs} we report on the EOCs, for the second order ESFEM coupled with the BDF3 method, theoretically we expect EOC$\approx3$ in the $L^\infty(L^2)$ norm, and EOC$\approx2$ in the $L^2(H^1)$ seminorm.
\begin{table}[!ht]
    \centering
    \begin{tabular}{r r  l l l l}
        \toprule
        level & dof & $L^\infty(L^2)$ & EOC & $L^2(H^1)$ & EOC \\
        \midrule
            $1$ & $6$     & $5.3113\cdot 10^{-3}$ & -        & $8.0694\cdot 10^{-3}$ & - \\
            $2$ & $18$    & $2.9257\cdot 10^{-3}$ & $0.9511$ & $4.3162\cdot 10^{-3}$ & $0.99805$ \\
            $3$ & $66$    & $9.2303\cdot 10^{-4}$ & $1.7122$ & $1.9050\cdot 10^{-3}$ & $1.2139$ \\
            $4$ & $258$   & $1.7285\cdot 10^{-4}$ & $2.4338$ & $5.4403\cdot 10^{-4}$ & $1.8207$ \\
            $5$ & $1026$  & $2.6463\cdot 10^{-5}$ & $2.7124$ & $1.2265\cdot 10^{-4}$ & $2.1529$ \\
            $6$ & $4098$  & $3.6845\cdot 10^{-6}$ & $2.8457$ & $2.4772\cdot 10^{-5}$ & $2.3088$ \\
        \bottomrule
    \end{tabular}
    \caption{Errors and EOCs in the $L^\infty(L^2)$ and $L^2(H^1)$ norms for the stationary problem}
    \label{table: stationary EOCs}
\end{table}

In Figure~\ref{fig: stationary1} and \ref{fig: stationary2} we report on the errors
\begin{equation*}
    \|u(\cdot,N \tau)-u_h^N\|_{L^2(\Ga)} \andquad \|\nbg(u(\cdot,N \tau)-u_h^N)\|_{L^2(\Ga)} .
\end{equation*}
at time $N\tau=1$. The logarithmic plots show the errors against the mesh width $h$ (in Figure~\ref{fig: stationary1}), and against time step size $\tau$ (in Figure~\ref{fig: stationary2}).

The different lines are corresponding to different time step sizes and to different mesh refinements, respectively in Figure~\ref{fig: stationary1} and \ref{fig: stationary2}.
In both figures we can observe two regions: In Figure~\ref{fig: stationary1}, a region where the spatial discretisation error dominates, matching to the convergence rates of our theoretical results, and a region, with fine meshes, where the time discretisation error is dominating (the error curves are flatting out). In Figure~\ref{fig: stationary2}, the same description applies, but with reversed roles. First the time discretisation error dominates, while for smaller stepsizes the spatial error is dominating.
The convergence in space (Figure~\ref{fig: stationary1}), and in time (Figure~\ref{fig: stationary2}) can both be nicely observed in agreement with the theoretical results (note the reference lines).


\begin{figure}[!hb]
    \centering
    \includegraphics[width=\textwidth,height=0.475\textheight]{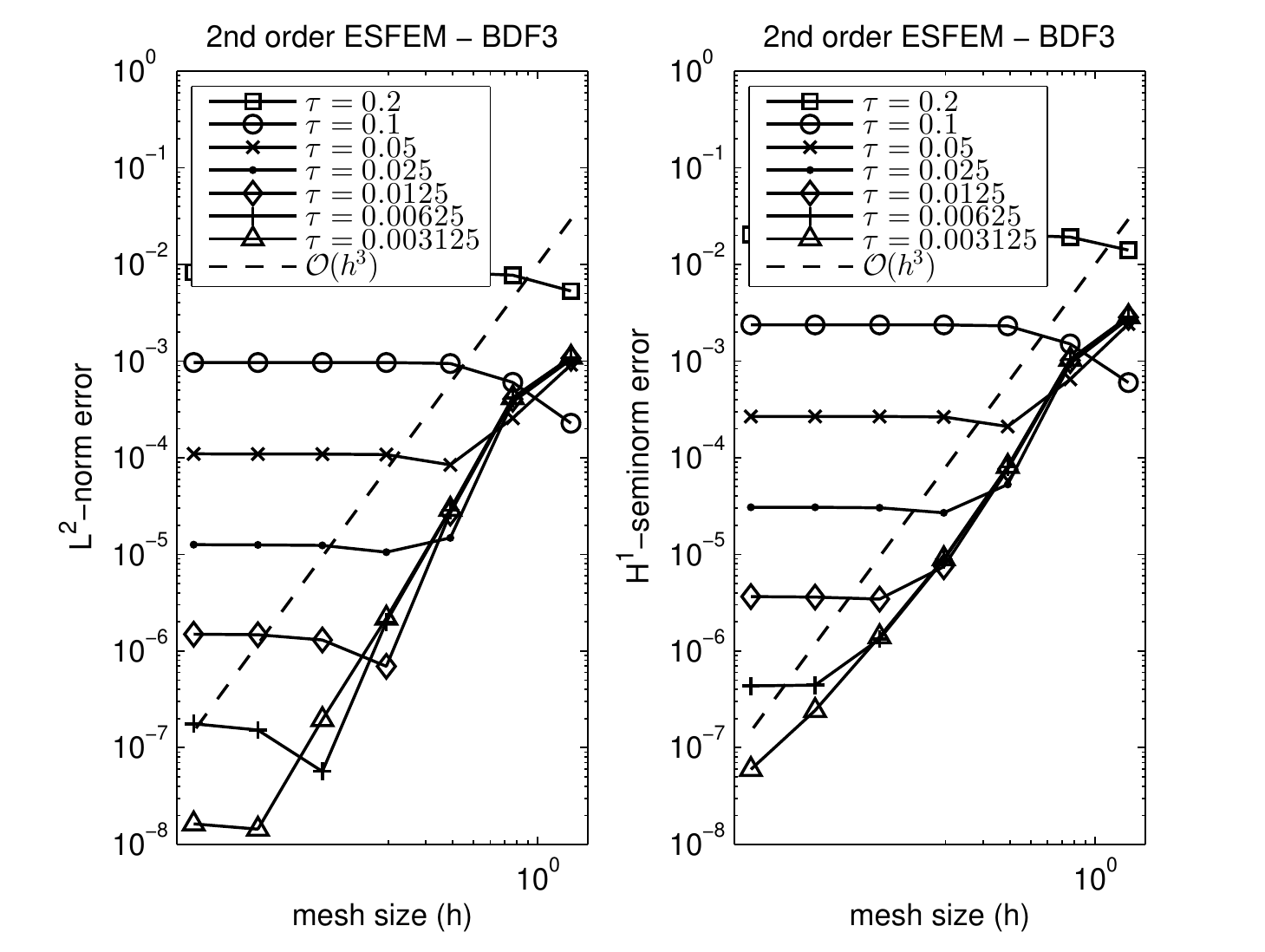}
    \caption{Spatial convergence of the BDF3 / quadratic SFEM discretisation for the stationary surface PDE}\label{fig: stationary1}
    \includegraphics[width=\textwidth,height=0.475\textheight]{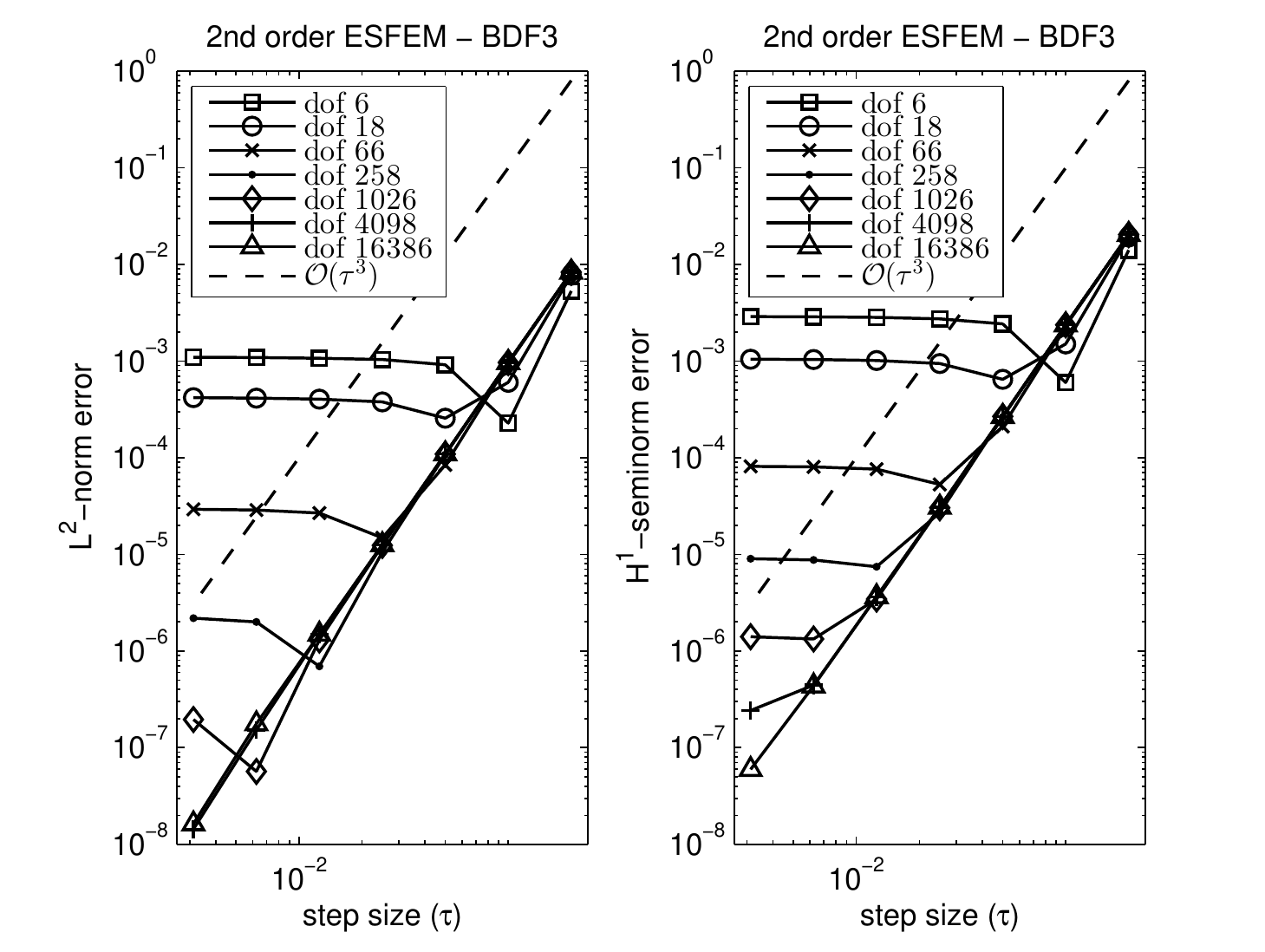}
    \caption{Temporal convergence of the BDF3 / quadratic SFEM discretisation for the stationary surface PDE}\label{fig: stationary2}
\end{figure}

\clearpage

\subsection{Example 2: Evolving surface parabolic problem}

In the following experiment we consider the parabolic problem \eqref{eq: strong form} on the evolving surface given by
\begin{equation*}
    \Ga\t = \big\{ x\in\R^3 \ \big| \ a(t)\inv x_1^2 + x_2^2 + x_3^2 - 1 = 0 \big\},
\end{equation*}
where $a(t)=1+\frac{1}{4}\sin(2\pi t)$, see, e.g.\ \cite{DziukElliott_ESFEM}, \cite{DziukLubichMansour_rksurf}, \cite{diss_Mansour}, with given initial value and inhomogeneity $f$ chosen such that the solution is $u(x,t)=e^{-6t}x_1x_2$.

Similarly to the stationary surface case we again report on experimental order of convergences and similar spatial and temporal convergence plots. They are all produced exactly as described above.

The EOCs for the evolving surface problem solved with BDF method of order $3$ and evolving surface finite elements of second order can be seen in Table \ref{table: evolving surface EOCs}.
\begin{table}[!h]
    \centering
    \begin{tabular}{r r  l l l l}
        \toprule
        level & dof & $L^\infty(L^2)$ & EOC & $L^2(H^1)$ & EOC \\
        \midrule
            1 & 6    & $5.7898\cdot 10^{-3}$ & -      & $8.4446\cdot 10^{-3}$ & -\\
            2 & 18   & $9.5840\cdot 10^{-4}$ & 2.8688 & $1.8173\cdot 10^{-3}$ & 2.4504\\
            3 & 66   & $4.0725\cdot 10^{-4}$ & 1.2704 & $1.6549\cdot 10^{-3}$ & 0.13895\\
            4 & 258  & $9.1096\cdot 10^{-5}$ & 2.1755 & $5.4513\cdot 10^{-4}$ & 1.6133\\
            5 & 1026 & $1.3847\cdot 10^{-5}$ & 2.7226 & $1.2774\cdot 10^{-4}$ & 2.0971\\
            6 & 4098 & $1.9534\cdot 10^{-6}$ & 2.8267 & $2.6135\cdot 10^{-5}$ & 2.2901\\
        \bottomrule
    \end{tabular}
    \caption{Errors and EOCs in the $L^\infty(L^2)$ and $L^2(H^1)$ norms for the evolving surface problem}
    \label{table: evolving surface EOCs}
\end{table}

The errors at time $N\tau=1$ in different norms can be seen in the following plots: the different lines are again corresponding to different time step sizes and to different mesh refinements, respectively in Figure~\ref{fig: evolving surface 1} and \ref{fig: evolving surface 2}.


\begin{figure}[!ht]
    \centering
    \includegraphics[width=\textwidth,height=0.475\textheight]{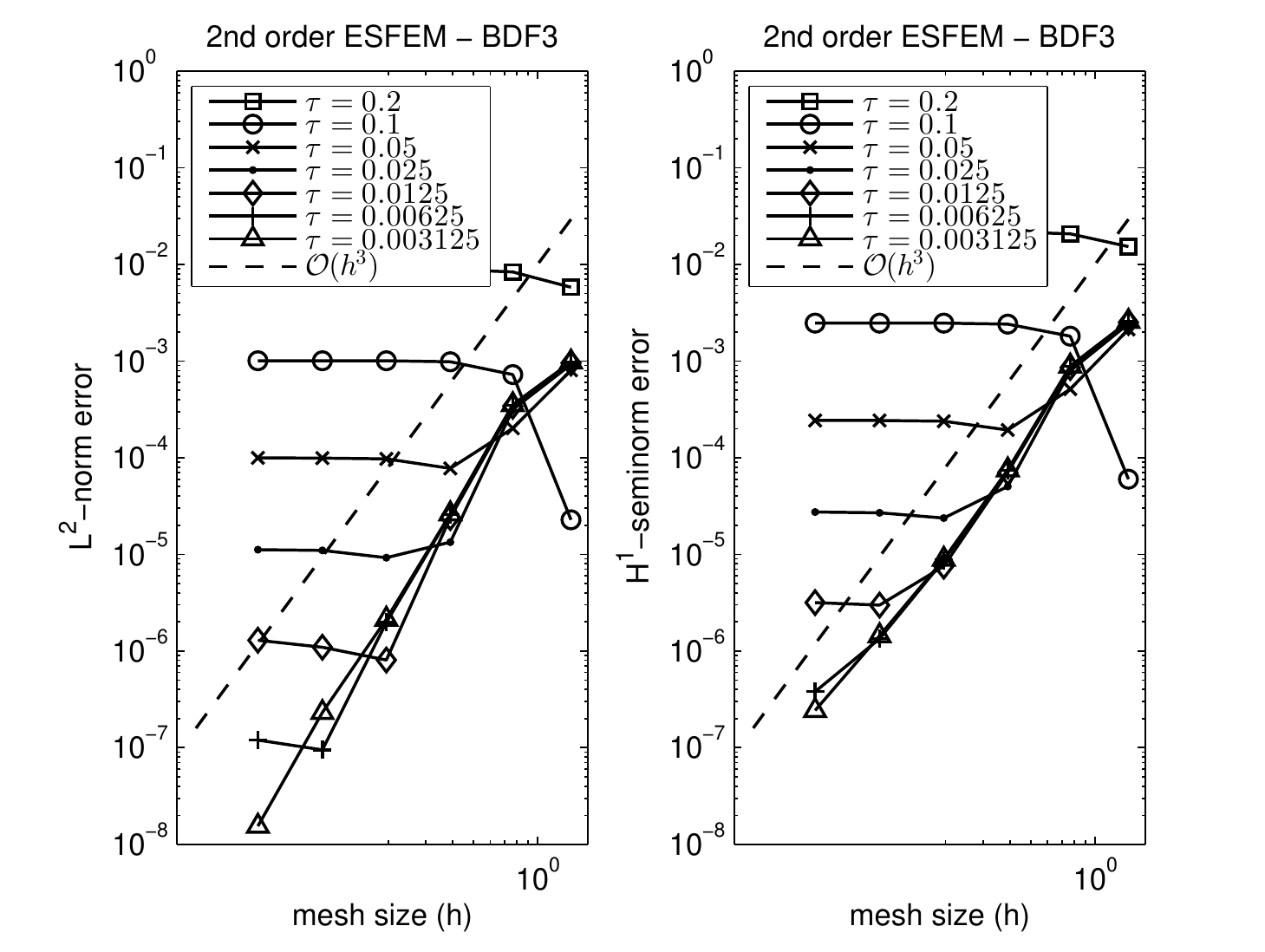}
    \caption{Spatial convergence of the BDF3 / quadratic ESFEM discretisation for the evolving surface PDE}\label{fig: evolving surface 1}
    \includegraphics[width=\textwidth,height=0.475\textheight]{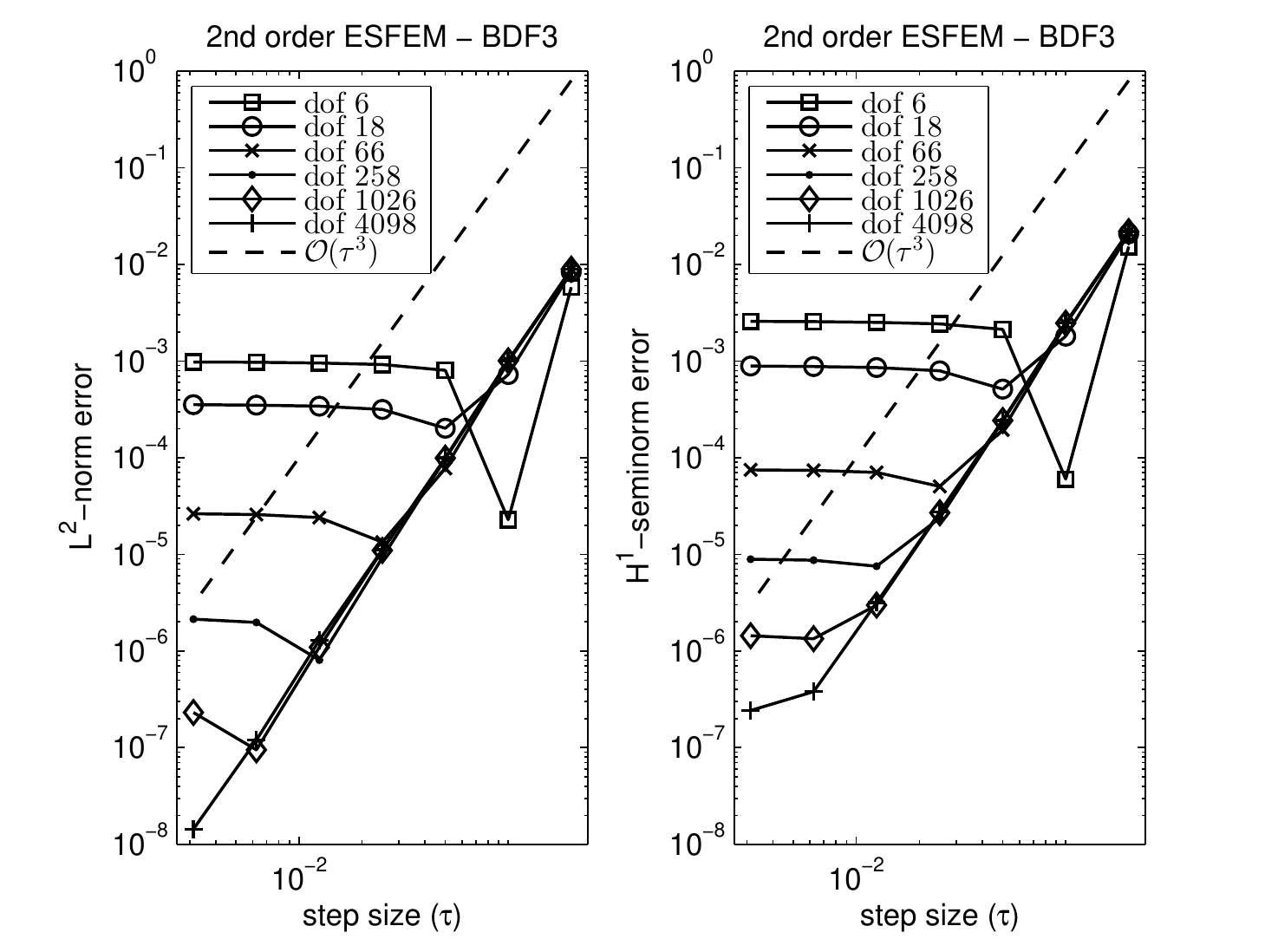}
    \caption{Temporal convergence of the BDF3 / quadratic ESFEM discretisation  for the evolving surface PDE}\label{fig: evolving surface 2}
\end{figure}

\clearpage
\subsection{Example 3: one dimensional experiments}

In the same fashion we also report on 1D numerical experiments, however we restrict ourselves only to present the corresponding tables and figures. The considered stationary and evolving surface PDEs are simply the one dimensional versions of the ones considered above.

\begin{table}[!ht]
    \centering
    \begin{tabular}{r r  l l l l}
        \toprule
        level & dof & $L^\infty(L^2)$ & EOC & $L^2(H^1)$ & EOC \\
        \midrule
            1 & 16 &   $6.1355\cdot 10^{-3}$ & - &    $6.5036\cdot 10^{-3}$ & - \\
            2 & 32 &   $2.0671\cdot 10^{-3}$ & 1.56 & $2.6949\cdot 10^{-3}$ & 1.27\\
            3 & 64 &   $4.1206\cdot 10^{-4}$ & 2.32 & $5.4668\cdot 10^{-4}$ & 2.30\\
            4 & 128 &  $6.4495\cdot 10^{-5}$ & 2.67 & $8.6709\cdot 10^{-5}$ & 2.65\\
            5 & 256 &  $9.0589\cdot 10^{-6}$ & 2.83 & $1.2207\cdot 10^{-5}$ & 2.82\\
            6 & 512 &  $1.2008\cdot 10^{-6}$ & 2.91 & $1.6195\cdot 10^{-6}$ & 2.91\\
            7 & 1024 & $1.5458\cdot 10^{-7}$ & 2.95 & $2.0856\cdot 10^{-7}$ & 2.95\\
            8 & 2048 & $1.9609\cdot 10^{-8}$ & 2.97 & $2.6462\cdot 10^{-8}$ & 2.97\\
        \bottomrule
    \end{tabular}
    \caption{1D: Errors and EOCs in the $L^\infty(L^2)$ and $L^2(H^1)$ norms for the stationary problem}
\end{table}

\begin{figure}[!hb]
    \centering
    \includegraphics[width=\textwidth,height=0.475\textheight]{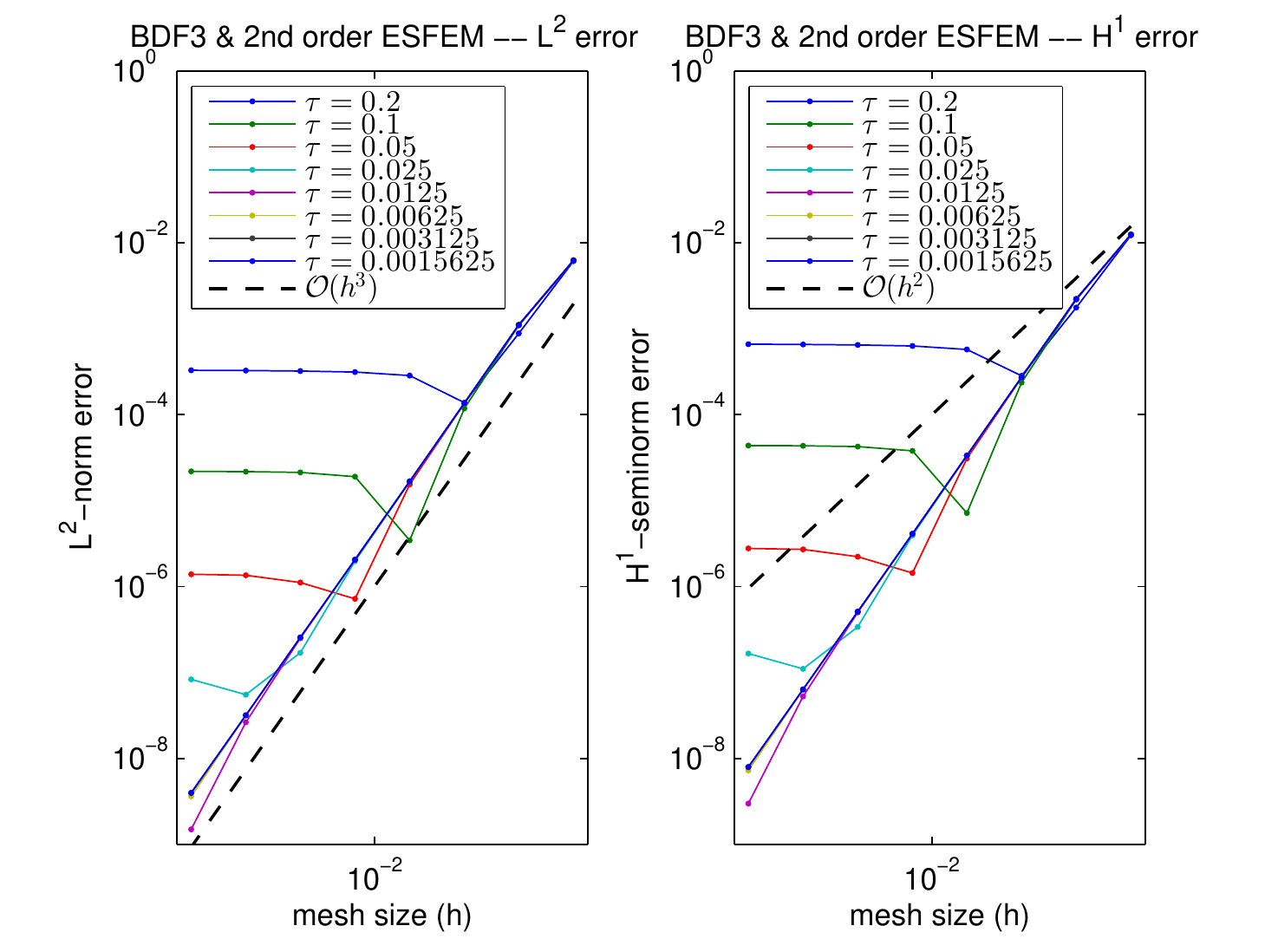}
    \caption{1D: Spatial convergence of the BDF3 / quadratic SFEM discretisation for the stationary surface PDE}
    \includegraphics[width=\textwidth,height=0.475\textheight]{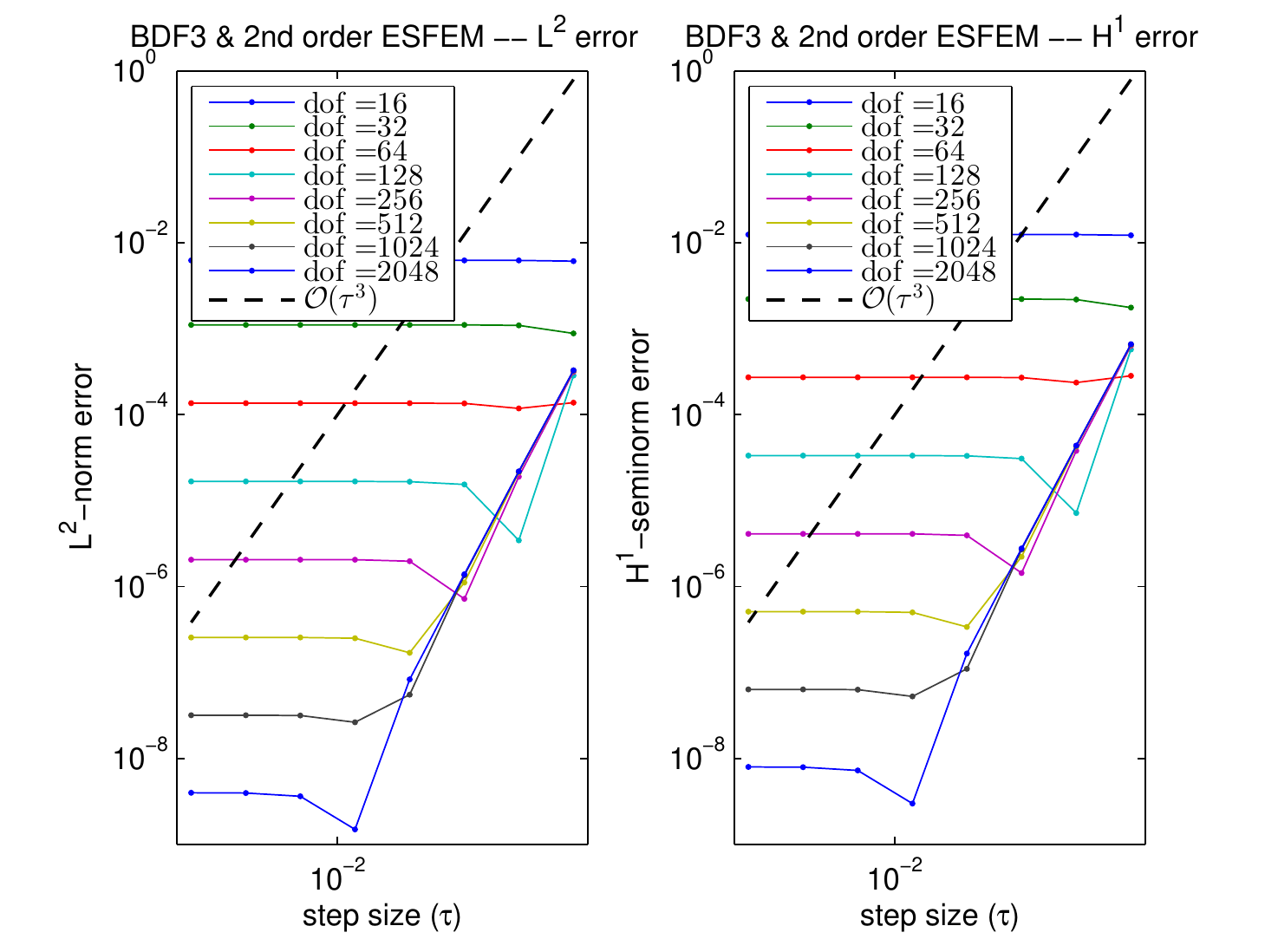}
    \caption{1D: Temporal convergence of the BDF3 / quadratic SFEM discretisation for the stationary surface PDE}
\end{figure}

\clearpage

\begin{table}[!h]
    \centering
    \begin{tabular}{r r  l l l l}
        \toprule
        level & dof & $L^\infty(L^2)$ & EOC & $L^2(H^1)$ & EOC \\
        \midrule
            1 & 16 &   $2.2866\cdot 10^{-3}$ & - &      $2.3829\cdot 10^{-3}$ & - \\
            2 & 32 &   $1.7292\cdot 10^{-3}$ & 0.40 & $1.9281\cdot 10^{-3}$ & 0.30\\
            3 & 64 &   $1.8684\cdot 10^{-4}$ & 3.21 & $2.2002\cdot 10^{-4}$ & 3.13\\
            4 & 128 &  $6.2647\cdot 10^{-5}$ & 1.57 & $4.9693\cdot 10^{-5}$ & 2.14\\
            5 & 256 &  $1.1521\cdot 10^{-5}$ & 2.44 & $9.6061\cdot 10^{-6}$ & 2.37\\
            6 & 512 &  $1.7156\cdot 10^{-6}$ & 2.74 & $1.4725\cdot 10^{-6}$ & 2.70\\
            7 & 1024 & $2.3306\cdot 10^{-7}$ & 2.87 & $2.0291\cdot 10^{-7}$ & 2.85\\
            8 & 2048 & $3.0341\cdot 10^{-8}$ & 2.94 & $2.6601\cdot 10^{-8}$ & 2.93\\
        \bottomrule
    \end{tabular}
    \caption{1D: Errors and EOCs in the $L^\infty(L^2)$ and $L^2(H^1)$ norms for the evolving surface problem}
\end{table}

\begin{figure}[!ht]
    \centering
    \includegraphics[width=\textwidth,height=0.475\textheight]{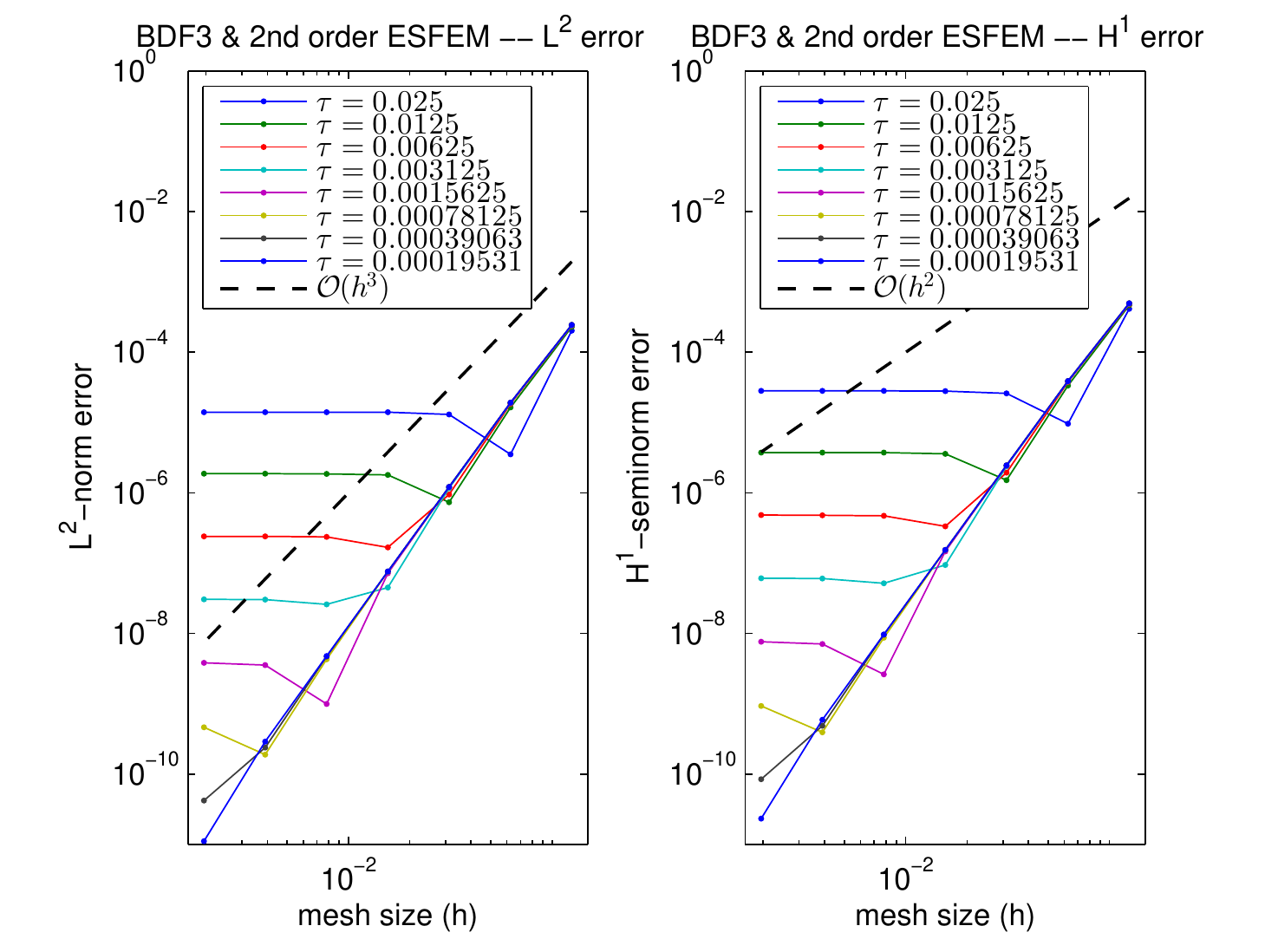}
    \caption{1D: Spatial convergence of the BDF3 / quadratic ESFEM discretisation for the evolving surface PDE}
    \includegraphics[width=\textwidth,height=0.475\textheight]{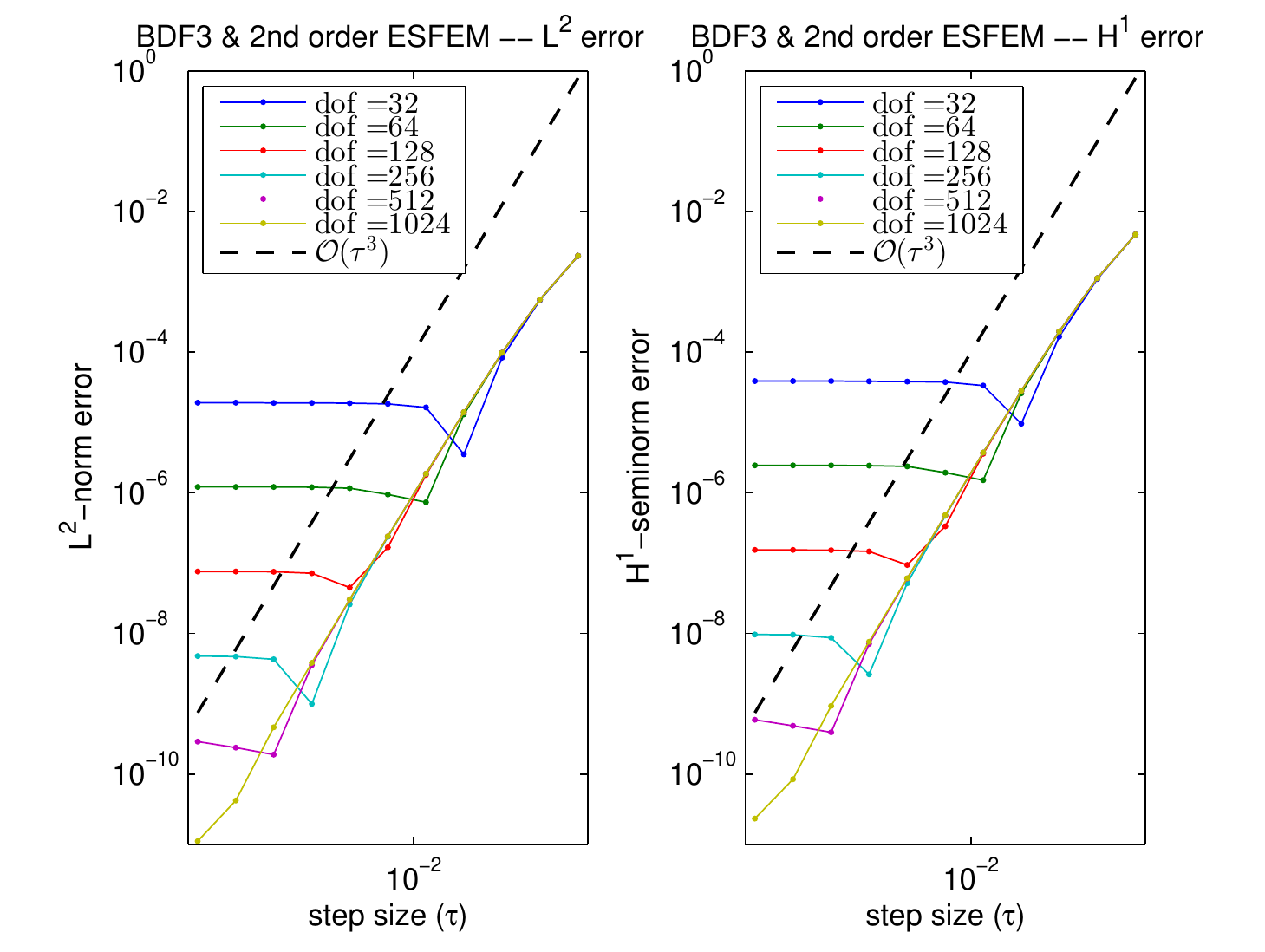}
    \caption{1D: Temporal convergence of the BDF3 / quadratic ESFEM discretisation  for the evolving surface PDE}
\end{figure}

\section*{Acknowledgement}
The author is grateful for the valuable discussions with Prof.~Christian Lubich during the preparation of the manuscript. The work of Bal\'azs Kov\'acs is funded by Deutsche Forschungsgemeinschaft, SFB 1173.

\clearpage


\end{document}